\documentclass[12pt,a4paper]{article}
\usepackage[T1]{fontenc}
\usepackage{amsfonts}
\usepackage{amssymb,amsmath,amsbsy,amsthm}
\IfFileExists{newtxtext.sty}
  {\usepackage{newtxtext,newtxmath}}
  {\usepackage{mathptmx}}
\usepackage{blkarray}
\usepackage{tikz}
\usepackage{tkz-euclide}
\usepackage{tikz-cd}
\usepackage{enumerate}
\usetikzlibrary{patterns,arrows.meta,calc,positioning,decorations.markings}
\usepackage{graphicx}
\usepackage{float}
\usepackage{cite}
\usepackage{mathtools}
\usepackage{indentfirst}
\usepackage{lineno}
\usepackage{titlesec}
\usepackage{enumitem}
\usepackage{xcolor}
\usepackage{colortbl}
\definecolor{tensorhighlight}{gray}{0.92}
\usepackage{aliascnt}
\usepackage{varioref}
\usepackage{microtype}
\usepackage{needspace}
\usepackage{xurl}
\usepackage{hyperref}
\hypersetup{
  colorlinks=true,
  linkcolor=cyan,
  anchorcolor=cyan,
  citecolor=red,
  urlcolor=black,
  pdfencoding=auto,
  pdftitle={Represented Tensor Products of Binary Matroids},
  pdfauthor={Houshan Fu, Yujiao Ma, Suijie Wang},
  pdfsubject={Classification of represented tensor products of binary matroids},
  pdfkeywords={represented tensor product, binary matroid, regular matroid, cographic matroid, series-parallel matroid, totally unimodular matrix},
}
\usepackage[nameinlink,capitalise,noabbrev]{cleveref}

\numberwithin{equation}{section}
\DeclareMathOperator{\rank}{rank}
\DeclareMathOperator{\col}{col}
\DeclareMathOperator{\diag}{diag}
\DeclareMathOperator{\row}{row}

\newcommand{\F}{\mathbb F}

\newcommand{\QQ}{\mathbb Q}

\newtheorem{theorem}{Theorem}[section]
\newtheorem*{maintheorem}{Main Theorem}
\newcommand{\mainthm}{\hyperref[thm:regular-two-factor]{Main Theorem}}
\newaliascnt{lemma}{theorem}
\newtheorem{lemma}[lemma]{Lemma}
\aliascntresetthe{lemma}

\newaliascnt{proposition}{theorem}
\newtheorem{proposition}[proposition]{Proposition}
\aliascntresetthe{proposition}

\newaliascnt{corollary}{theorem}
\newtheorem{corollary}[corollary]{Corollary}
\aliascntresetthe{corollary}

\newaliascnt{claim}{theorem}

\aliascntresetthe{claim}

\newaliascnt{conjecture}{theorem}

\aliascntresetthe{conjecture}

\theoremstyle{definition}
\newtheorem{fact}{Fact}[section]
\crefname{fact}{Fact}{Facts}
\Crefname{fact}{Fact}{Facts}
\newaliascnt{definition}{theorem}

\aliascntresetthe{definition}

\newaliascnt{example}{theorem}
\newtheorem{example}[example]{Example}
\aliascntresetthe{example}

\newaliascnt{notation}{theorem}

\aliascntresetthe{notation}

\newaliascnt{question}{theorem}

\aliascntresetthe{question}

\theoremstyle{remark}
\newaliascnt{remark}{theorem}
\newtheorem{remark}[remark]{Remark}
\aliascntresetthe{remark}

\tikzset{
  graphvertex/.style={circle,fill,inner sep=1.35pt},
  leafvertex/.style={circle,draw=black,fill=black,inner sep=1.35pt},
  facevertex/.style={circle,draw,fill=white,inner sep=1.5pt,font=\scriptsize},
  graphlabel/.style={font=\scriptsize,inner sep=1pt}
}
\begin{document}
\begin{center}
{\Large\bf
Represented Tensor Products of Binary Matroids}\\[10pt]
\end{center}

\begin{center}
Houshan Fu\quad Yujiao Ma\quad Suijie Wang
\end{center}

\begin{abstract}
For binary matroids \(M,N\) representable over a common field
\(\F\), the Kronecker product of their \(\F\)-representations defines a
matroid \(T_{\F}(M,N)\) independent of the chosen representations.
We classify when this represented tensor product is regular, cographic,
graphic, or binary.  For simple nonfree factors, regularity holds exactly
when, up to interchange, one factor is a cactus matroid and the other
is outerplanar, or one is a triangular cactus matroid and the other
is series--parallel.  Cographicity holds exactly in the first case.
For simple factors with nonempty ground sets, graphicity holds exactly
when one factor is free and the other is graphic.
For fixed factors, regularity, cographicity, and graphicity are
independent of the common representation field, although the
isomorphism type may vary with its characteristic.
Products of at least three simple nonfree factors are nonregular.
\vspace{1ex}
\vspace{1ex}
\end{abstract}

\section{Introduction}\label{sec:introduction}
Let \(M,N\) be finite binary matroids representable over a common
field \(\F\).
Choose \(\F\)-representations \(A\) and \(B\), with columns labelled
by \(E(M)\) and \(E(N)\), respectively.  We study the matroid
\[
 T_{\F}(M,N):=M[A\otimes B]
\]
on \(E(M)\times E(N)\), where the column indexed by \((e,f)\) is
\(\boldsymbol a_e\otimes\boldsymbol b_f\).  Binary matroids are
uniquely representable over each representation field up to projective equivalence
\cite{brylawski-lucas}; see also \cite[Proposition~6.6.5]{oxley}.
Consequently, the product does not depend on the chosen
representations.  We call this construction the \emph{represented tensor product}.
It is a tensor product in the abstract sense of Las Vergnas
\cite{las-vergnas-products}; see Section~\ref{subsec:tensor-preliminaries}.
The common-field hypothesis allows arbitrary binary factors in
characteristic two, whereas in other characteristics both factors must
be regular \cite[Theorem~6.6.3]{oxley}.
Our aim is to characterize when it is regular, cographic, graphic, or binary.

To state the main result, call a simple matroid a \emph{cactus matroid}
if each connected component is a circuit matroid or a single coloop.
A cactus matroid is \emph{triangular} if every circuit component has
size three.  A simple binary matroid is called \emph{series--parallel} if
it has no \(M(K_4)\)-minor, and \emph{outerplanar} if it has neither
an \(M(K_4)\)- nor an \(M(K_{2,3})\)-minor.  These terms agree with
the corresponding graph classes, as shown in
Lemma~\ref{lem:regular-sp-outerplanar}.  A free matroid is one in
which every subset is independent.

\begin{maintheorem}\phantomsection\label{thm:regular-two-factor}
Let \(M,N\) be simple binary matroids on nonempty ground sets,
representable over a common field \(\F\).  The following statements hold.
\begin{enumerate}[label=\textup{(\roman*)},leftmargin=2.5em]
\item \(T_{\F}(M,N)\) is regular if and only if either one factor is
free and the other is regular, or both are nonfree and, after possibly exchanging them,
one of the following holds:
\[
 \begin{array}{ll}
 \textup{(R1)}&M\text{ is a cactus matroid and }N\text{ is outerplanar},\\
 \textup{(R2)}&M\text{ is a triangular cactus matroid and }
                  N\text{ is series--parallel}.
 \end{array}
\]
\item \(T_{\F}(M,N)\) is cographic if and only if either one factor
is free and the other is cographic, or both are nonfree and
\textup{(R1)} holds after possibly exchanging the factors.
\item \(T_{\F}(M,N)\) is graphic if and only if one factor is free
and the other is graphic.
\item \(T_{\F}(M,N)\) is binary if and only if
\(\operatorname{char}(\F)=2\), or \(T_{\F}(M,N)\) is regular.
\end{enumerate}
\end{maintheorem}

For fixed \(M,N\), parts~\textup{(i)--(iii)} show that regularity,
cographicity, and graphicity of \(T_{\F}(M,N)\) are independent of
the common representation field.  Binarity is automatic in characteristic
two and, by part~\textup{(iv)}, is also independent of the choice of
common representation field of characteristic different from two.
Products over fields of the same characteristic agree by
Lemma~\ref{lem:prime-field-reduction}.  Their isomorphism type may
vary with the characteristic; see
Remark~\ref{ex:regular-characteristic}.

Las Vergnas developed the abstract theory and showed that
not every pair of matroids admits a tensor product
\cite[Proposition~2.1]{las-vergnas-products}.  The tensor product of
vector representations always gives an abstract tensor product
\cite[p.~50]{las-vergnas-products}.  Anderson expressed this construction
using Kronecker matrices and studied its relation to tropical geometry,
including the possible dependence on representations for general
representable factors \cite[Section~2.1]{anderson-matroid-products}.
B\'erczi, Geh\'er, Imolay, Lov\'asz, Maga, and Schwarcz studied matroid
products through submodular coupling \cite{berczi-matroid-products}.
B\'erczi, Geh\'er, Imolay, Lov\'asz, Padr\'o, and Schwarcz developed
connections with skew-representability, extension properties, and
rank inequalities \cite{berczi-skew-tensor}; Lemma~7.14 of their
expanded version identifies the tensor product of \(U_{2,3}\) with
itself as \(M^*(K_{3,3})\).

Incidence Kronecker matrices were studied by Hanusa and Zaslavsky
\cite{hanusa-zaslavsky}.  Plesken and B\"achler considered the matroid
represented by
\([I_4\mid-\boldsymbol1_4]\otimes[I_2\mid-\boldsymbol1_2]\)
\cite[Example~4.13]{plesken-baechler}, a represented product of the
cycle matroids of \(C_5\) and \(C_3\).
Martin and Reiner \cite[Theorem~1 and Section~2]{martin-reiner}
relate represented products of circuits of distinct prime sizes over
\(\QQ\) to duals of complete multipartite simplicial matroids.
Hidaka and Itoh \cite[Theorem~1.2 and Section~5.2]{hidaka-itoh}
prove total unimodularity of the standard tensor representation of
two circuits and record its cographic interpretation in characteristic
zero.  Related constructions include higher order independence and symmetric
powers of vector configurations \cite{baclawski-white}, tensor and
symmetric tensor matroids arising in rigidity and matrix completion
\cite{brakensiek-rigidity,jackson-tanigawa-symmetric-tensor}, and
polymatroid tensor products \cite{padro-polymatroid-tensor}.
Tyomkyn \cite{tyomkyn-linear-products} studies linearly representable
tensor products of uniform matroids through their duality with abstract
birigidity matroids.

When both factors have positive rank, each is isomorphic to a restriction
of the product, so a regular product has regular factors.
Every simple nonfree matroid has a \(U_{2,3}\)-minor.  By the
factor-minor property and Proposition~\ref{prop:specified-minors}\textup{(i)},
regularity of the product of two such factors excludes an \(M(K_4)\)-minor
in either factor.  Both are therefore series--parallel graphic matroids
\cite{brylawski-series-parallel,bonin-long}.
The reduction and the remaining necessity arguments use the minor
containments in Proposition~\ref{prop:specified-minors}\textup{(i)--(ii)}
for the pairs \((C_3,K_4)\), \((K_4-e,K_4-e)\),
\((C_4,K_{2,3})\), and \((C_3,K_{2,3})\).
These containments, proved in Section~\ref{sec:minor-proofs}, also yield
the forbidden-factor characterization in
Corollary~\ref{cor:regular-forbidden-pairs}.

For graphs \(G,H\), write
\(M_{\F}(G,H)=T_{\F}(M(G),M(H))\).  For \(m\ge3\) and a
\(2\)-connected outerplanar simple graph \(G\), we construct a graph \(X_{m,G}\) with
\(M_{\F}(C_m,G)\cong M^*(X_{m,G})\).  For two cycles, this recovers
the complete bipartite graph realization:
\[
 M_{\F}(C_m,C_n)\cong M^*(K_{m,n})\qquad(m,n\ge3).
\]
For a triangle and a series--parallel network, edge subdivisions and
parallel-edge additions yield a totally unimodular tensor
representation.  These constructions give the two nonfree regular cases in
the \mainthm{}.

A free factor contributes only direct sum multiplicities.  For simple
factors with nonempty ground sets, a product with at least three nonfree
factors has
\(T_{\F}(U_{2,3},U_{2,3},U_{2,3})\) as a minor and is nonregular by
Proposition~\ref{prop:specified-minors}\textup{(iii)}.  This gives
Theorem~\ref{thm:regular-higher};
Proposition~\ref{prop:regular-simplification} then extends the
classification to factors with loops or parallel elements.

Section~\ref{sec:preliminaries} collects the preliminaries.
Section~\ref{sec:positive-constructions} develops the tools for the
classification: the reduction to graphic factors, the required minor
containments, and the cographic and totally unimodular constructions.
Section~\ref{sec:regular-matroids} gives the classifications,
including the proof of the \mainthm{} in
Section~\ref{subsec:regular-two-factor-proof}.
Section~\ref{sec:minor-proofs} proves the minor containments stated
in Section~\ref{sec:positive-constructions}.

\section{Preliminaries and basic properties}\label{sec:preliminaries}
Throughout the paper, write \([n]=\{1,\ldots,n\}\) for a positive
integer \(n\).  For a matrix \(A\) over a field \(\F\), let
\(\rank(A)\), \(\ker(A)\), \(\row(A)\), and \(\col(A)\)
denote its rank, kernel, row space, and column space, respectively.
If the columns of \(A\) are labelled by \(E\), then \(A_X\)
denotes the submatrix on the columns indexed by \(X\subseteq E\).

Integer matrices are interpreted over \(\F\) via the canonical
homomorphism \(\mathbb Z\to\F\); we write \(A^{\F}\) when
the field needs to be indicated.  We use \(I_n\) for the identity
matrix, \(\boldsymbol{0}_n,\boldsymbol{1}_n\in\F^n\) for the
all-zero and all-one vectors, and \(\boldsymbol{e}_i^{(n)}\) for
the \(i\)-th standard basis vector, omitting the superscript when
the dimension is clear.

\subsection{Matroid and graph preliminaries}\label{subsec:matroid-preliminaries}
All matroids considered are finite.  The terminology follows
Oxley \cite{oxley}.  For a matroid \(M\),
write \(E(M)\) for its ground set, \(r_M\) for its rank function,
\(r(M)=r_M(E(M))\), and \(M^*\) for its dual.  We use
\(M|X\), \(M\backslash X\), and \(M/X\) for restriction, deletion,
and contraction.  A \emph{minor} is obtained by deletions and contractions.
We write \(M_1\oplus M_2\) for a direct sum, \(U_{r,n}\) for the
rank-\(r\) uniform matroid on \(n\) elements, and \(F_7\) for the
Fano matroid.  For \(n\ge0\),
\(N^{\oplus n}\) is the direct sum of \(n\) copies of \(N\) on disjoint
ground sets, with \(N^{\oplus0}\) the empty matroid.  A pair
\((N_1,N_2)\) occurs as \emph{factor minors} of \((M_1,M_2)\) if
\(N_i\) is isomorphic to a minor of \(M_i\) for \(i=1,2\).

A circuit is a minimal dependent set, and a cocircuit is a circuit of
the dual.  An element is a loop or a coloop if its singleton is a
circuit or a cocircuit, respectively.  Two distinct elements are
parallel or in series if they form a circuit or a cocircuit, respectively.
A matroid is \emph{simple} if it has no loops or parallel elements,
and \emph{free} if it is \(U_{q,q}\) for some \(q\ge0\).
Its \emph{simplification} \(\operatorname{si}(M)\) deletes all loops
and retains one element from each nonloop parallel class; we regard
it as a restriction on specified representatives.

A nonempty matroid is \emph{connected} if it is not a direct sum of
two nonempty matroids.  Every matroid is the direct sum of its
connected components \cite[Corollary~4.2.9]{oxley}.  For
\(e\in E(M)\) and \(p\notin E(M)\), a matroid \(N\) on
\(E(M)\cup\{p\}\) is a
\emph{parallel extension} of \(M\) at \(e\) if
\(N\backslash p=M\) and \(\{e,p\}\) is a circuit; it is a
\emph{series extension} at \(e\) if \(N/p=M\) and \(\{e,p\}\)
is a cocircuit \cite[Section~5.4, pp.~151--152]{oxley}.

For a matrix \(A=[\boldsymbol a_e:e\in E]\) over \(\F\), let
\(M[A]\) be its vector matroid on the labelled column set \(E\).
A matroid is \emph{graphic} if it is isomorphic to the cycle matroid
\(M(G)\) of a graph \(G\), and \emph{cographic} if it is isomorphic
to \(M^*(G)\) for some graph \(G\).  It is \emph{binary} if representable over
\(\operatorname{GF}(2)\) and \emph{regular} if representable over
every field.  A real matrix is \emph{totally unimodular} (TU) if all
its square subdeterminants lie in \(\{0,1,-1\}\); a matroid is
regular exactly when it has a TU real representation
\cite[Lemma~2.2.21 and Theorem~6.6.3]{oxley}.

A \emph{cactus matroid} is a simple matroid whose connected components are single
coloops \(U_{1,1}\) or circuit matroids \(U_{m-1,m}\) with \(m\ge3\).  It is
\emph{triangular} if all its circuit components are \(U_{2,3}\).
A simple binary matroid is \emph{series--parallel} if it has no
\(M(K_4)\)-minor, and \emph{outerplanar} if it also has no
\(M(K_{2,3})\)-minor.  These graph-related terms are justified in
Lemma~\ref{lem:regular-sp-outerplanar}.

All graphs are finite.  Graphs may have loops and parallel edges unless
stated otherwise.  For a graph \(G\), write \(V(G)\), \(E(G)\),
\(c(G)\), and \(M(G)\) for its vertex set, edge set, number of
components, and cycle matroid, respectively.  Put
\(r(G):=r(M(G))=|V(G)|-c(G)\) and let \(M^*(G)\) be its bond
matroid.  A \emph{\(2\)-connected} graph is \(2\)-vertex-connected
and has at least three vertices.  A \emph{block} is a maximal
connected subgraph with connected cycle matroid; in a simple graph,
each block with an edge is a maximal \(2\)-connected subgraph or a bridge.
A \emph{cyclic block} contains a cycle.  A \emph{cactus} is a graph
whose blocks with edges are single edges or cycles.  For simple graphs this
means that distinct cycles have at most one vertex in common.

An \emph{outerplanar} graph admits a plane embedding with every
vertex on the outer-face boundary; such an embedding is
\emph{outerplane}.  An \emph{edge subdivision} replaces an edge by
a two-edge path through a new vertex (a loop by a two-edge cycle),
and a \emph{parallel-edge addition} duplicates a nonloop edge.
Subdividing a nonbridge edge induces a series extension of its cycle
matroid, while subdividing a bridge does not; a graphic series
extension need not be an edge subdivision in a chosen graph
representation \cite[pp.~151--152 and Figure~5.13]{oxley}.
A \emph{series--parallel network} is the single-edge graph or a
graph obtained from a two-edge cycle by edge subdivisions and
parallel-edge additions \cite[Section~5.4, Exercise~8]{oxley}.

We use the following standard facts from Oxley
\cite[Chapters~3--6 and Section~12.2, Exercise~14(a)]{oxley}.
For the equivalent forbidden-subdivision characterization of outerplanar
graphs, see also Chartrand and Harary
\cite[Theorem~1]{chartrand-harary}.

\begin{fact}\label{fact:graph-minors}
For graphs \(G,H\), if \(H\) is a graph minor of \(G\), then
\(M(H)\) is isomorphic to a minor of \(M(G)\).
\end{fact}

\begin{fact}\label{fact:matroid-classes}
Every graphic or cographic matroid is regular, and every regular
matroid is binary.  Every binary matroid is representable over every
field of characteristic two.  A matroid representable over a field
of characteristic different from two is regular if and only if it is
binary.
\end{fact}

\begin{fact}\label{fact:minor-direct-sums}
The classes of graphic, cographic, binary, and regular matroids
are closed under minors and finite direct sums.
\end{fact}

\begin{fact}\label{fact:binary-excluded}
A matroid is binary if and only if it has no \(U_{2,4}\)-minor.
\end{fact}

\begin{fact}\label{fact:regular-excluded}
A matroid is regular if and only if it has no minor isomorphic to
\(U_{2,4}\), \(F_7\), or \(F_7^*\).
\end{fact}

\begin{fact}\label{fact:cographic-excluded}
A matroid is cographic if and only if it has no minor isomorphic to
\(U_{2,4}\), \(F_7\), \(F_7^*\), \(M(K_5)\), or \(M(K_{3,3})\).
\end{fact}

\begin{fact}\label{fact:whitney-planarity}
A graph \(G\) is planar if and only if \(M^*(G)\) is graphic;
equivalently, \(M(G)\) is cographic if and only if \(G\) is planar.
\end{fact}

\begin{fact}\label{fact:series-parallel-minors}
For a \(2\)-connected simple graph \(G\),
\[
G\text{ is a series--parallel network}
 \quad\Longleftrightarrow\quad
 M(G)\text{ has no }M(K_4)\text{-minor}.
\]
\end{fact}

\begin{fact}\label{fact:outerplanar-minors}
For every finite simple graph \(G\),
\[
G\text{ is outerplanar}
 \quad\Longleftrightarrow\quad
 G\text{ has neither }K_4\text{ nor }K_{2,3}\text{ as a graph minor}.
\]
\end{fact}

\subsection{Represented tensor products}\label{subsec:tensor-preliminaries}
Las Vergnas introduced abstract tensor products of matroids in \cite{las-vergnas-products}; we use the equivalent formulation recorded
by Anderson \cite[Definition~2.1 and Lemma~2.2]{anderson-matroid-products}.  Let \(M\) and \(N\) be
matroids on ground sets \(E(M)\) and \(E(N)\).  A matroid \(P\) on
\(E(M)\times E(N)\) is a \emph{quasi-product} of \(M\) and \(N\) if, for each
nonloop \(e\in E(M)\), the restriction of \(P\) to
\(\{e\}\times E(N)\) is isomorphic to \(N\) under \(x\mapsto(e,x)\), and, for
each nonloop \(f\in E(N)\), the restriction of \(P\) to
\(E(M)\times\{f\}\) is isomorphic to \(M\) under \(x\mapsto(x,f)\).
In addition, fibres over loops of either factor are required to have rank zero.  Such a
quasi-product \(P\) is a \emph{tensor product} of \(M\) and \(N\) precisely
when
\[
r_P(X\times Y)=r_M(X)r_N(Y),\qquad X\subseteq E(M),\;Y\subseteq E(N).
\]
Equivalently, a quasi-product is a tensor product when \(r(P)=r(M)r(N)\).
This definition involves no field or choice of representation, and
such a product need not exist for arbitrary matroids
\cite[Proposition~2.1]{las-vergnas-products}.

Let \(A=[\boldsymbol{a}_e:e\in E]\) and \(B=[\boldsymbol{b}_f:f\in F]\)
be matrices over \(\F\).  Their Kronecker product \(A\otimes B\)
has columns labelled by \(E\times F\), with column
\(\boldsymbol{a}_e\otimes\boldsymbol{b}_f\) at \((e,f)\).
We call \(P=M[A\otimes B]\) the \emph{represented tensor product}
associated with \(A\) and \(B\).  For \(X\subseteq E\) and
\(Y\subseteq F\), the submatrix on \(X\times Y\) is
\(A_X\otimes B_Y\), so
\[
r_P(X\times Y)=\rank(A_X)\rank(B_Y)=r_{M[A]}(X)r_{M[B]}(Y).
\]
Consequently, \(P\) is an abstract tensor product of \(M[A]\) and
\(M[B]\).  In particular, the Cartesian product of bases of the two
factors is a basis of \(P\), as in Las Vergnas's original formulation
\cite[p.~50]{las-vergnas-products}.
For general representable factors, the resulting matroid may depend
on the chosen representations, even over a fixed field
\cite[Section~2.1]{anderson-matroid-products}.

Two matrices \(A\) and \(A'\) over \(\F\), with columns labelled
by the same set \(E\), are \emph{projectively equivalent} if they are
related by elementary row operations, insertion or deletion of zero
rows, and nonzero column scalings.  When both have \(r\) rows and
the same column order, this is equivalent to \(A'=XAY\), where
\(X\) is a nonsingular \(r\times r\) matrix and \(Y\) is a
nonsingular diagonal matrix \cite[p.~176 and Proposition~6.3.12]{oxley}.
Columns may be reordered together with their labels.  Whenever a
ground-set bijection is specified, the asserted matroid isomorphism
is induced by that bijection.

Let \(s\ge2\), and let \(M_1,\ldots,M_s\) be finite binary matroids
representable over a common field \(\F\).
For labelled \(\F\)-representations \(A_1,\ldots,A_s\), define
\begin{equation}\label{eq:regular-tensor-definition}
 T_{\F}(M_1,\ldots,M_s)
 :=M[A_1\otimes\cdots\otimes A_s].
\end{equation}
The ground set is \(E(M_1)\times\cdots\times E(M_s)\).
The definition includes loops and parallel elements.  For two factors
we write \(T_{\F}(M,N)\).

\begin{proposition}
\label{thm:basic-represented-product}
The labelled matroid in \eqref{eq:regular-tensor-definition} is
independent of the chosen factor representations.  Permuting the
factors gives the corresponding natural matroid isomorphism.
The canonical associativity isomorphisms of vector-space tensor
products preserve the resulting labelled matroid.
\end{proposition}
\begin{proof}
Row reduction followed by deletion of zero rows replaces each factor
matrix by a full row rank matrix without changing the represented
product.  Indeed, invertible row operations in one coordinate induce
invertible row operations in the tensor product, and a zero row
contributes only zero tensor rows.  If some factor has rank zero,
every tensor column is zero, so the assertion is immediate.

Otherwise, let \(R_i,S_i\) be two full row rank representations of
\(M_i\) with the same column labels.  The unique-representation theorem
for binary matroids
\cite{brylawski-lucas} (see also \cite[Proposition~6.6.5]{oxley}) gives
\(R_i=X_iS_iD_i\), where \(X_i\) is nonsingular and \(D_i\) is
nonsingular diagonal.  Hence
\[
 \bigotimes_{i=1}^s R_i=
 \left(\bigotimes_{i=1}^s X_i\right)
 \left(\bigotimes_{i=1}^s S_i\right)
 \left(\bigotimes_{i=1}^s D_i\right).
\]
The left and right multipliers perform invertible row operations and
nonzero column scalings, respectively, so they preserve the linear
independence of every set of columns.  The canonical permutation and associativity maps for
vector-space tensor products prove the remaining assertions.
\end{proof}

\begin{lemma}\label{lem:prime-field-reduction}
Let \(\F_0\) be the prime subfield of \(\F\), identified with
\(\QQ\) if \(\operatorname{char}(\F)=0\), and with
\(\operatorname{GF}(p)\) if \(\operatorname{char}(\F)=p>0\).
Let \(s\ge2\), and let \(M_1,\ldots,M_s\) be finite binary matroids
representable over \(\F\).  They are also representable over \(\F_0\),
and, under the common product labelling,
\begin{equation}\label{eq:prime-field-invariance}
 T_{\F}(M_1,\ldots,M_s)=T_{\F_0}(M_1,\ldots,M_s).
\end{equation}
Consequently, products over fields of the same characteristic agree,
and every product in characteristic two is binary.
\end{lemma}
\begin{proof}
If \(\operatorname{char}(\F)=2\), choose a
\(\operatorname{GF}(2)\)-representation \(A_i\) of each \(M_i\).
Otherwise, each \(M_i\) is regular by
Fact~\ref{fact:matroid-classes}; choose an integral TU representation
\(A_i\) \cite[Theorem~6.6.3]{oxley} and interpret it over \(\F_0\).
Its square minors lie in \(\{0,1,-1\}\), so it represents
\(M_i\) over \(\F_0\).

In both cases, extending scalars from \(\F_0\) to \(\F\) preserves
every column rank of each \(A_i\) and of
\(A_1\otimes\cdots\otimes A_s\), since a field inclusion
preserves precisely which minors vanish.  Proposition~\ref{thm:basic-represented-product}
now gives \eqref{eq:prime-field-invariance}.
The remaining assertions follow from the description of \(\F_0\).
\end{proof}

When all factors are regular, their TU representations may be used
over every field, but the Kronecker product need not be TU.
Its minors may vanish in one characteristic and remain nonzero in
another.  This possible dependence on the field was also noted in
\cite[Remark~4.15 of the expanded version]{berczi-skew-tensor}.
Remark~\ref{ex:regular-characteristic} illustrates both this
dependence on the characteristic and the nonuniqueness of abstract
tensor products, even for regular factors.  If some binary factor is
nonregular, every common representation field has characteristic two,
so Lemma~\ref{lem:prime-field-reduction} gives the same labelled
product over all such fields.

The associativity assertion concerns the representing vector spaces.
In characteristic two, every represented product of binary factors is
binary by Lemma~\ref{lem:prime-field-reduction}, so the operation
is also associative on binary matroids.  In other characteristics an
intermediate product need not be binary; higher products are therefore
defined directly by \eqref{eq:regular-tensor-definition}.

For matrices \(A,B\) with columns labelled by \(E,F\), respectively,
the natural identification
\(\F^E\otimes_{\F}\F^F\cong\F^{E\times F}\) gives
\[
\row(A\otimes B)=\row(A)\otimes_{\F}\row(B).
\]
Matrices with the same row space have the same column dependencies,
so this tensor subspace determines \(M[A\otimes B]\).
The graph constructions in Section~\ref{sec:positive-constructions}
use this row-space description.

For finite graphs \(G_1,\ldots,G_s\), we use the abbreviation
\[
M_{\F}(G_1,\ldots,G_s)
 :=T_{\F}(M(G_1),\ldots,M(G_s)).
\]
Graphic matroids are regular, so all the preceding conclusions apply.
In particular, any incidence representations of the cycle matroids
may be used to compute this product.

\subsection{Minors, direct sums, and simplification}
\label{subsec:minors-direct-sums}
The following properties hold for binary factors representable over
a common field.
For abstract tensor products, a special contraction argument appears
in \cite[Remark~2.3]{las-vergnas-products}, and the general
minor property is proved in
\cite[Proposition~2.7]{anderson-matroid-products}; see also
\cite[Lemma~4.1 of the expanded version]{berczi-skew-tensor}.
We give the matrix argument because it identifies the induced
representations explicitly.

\begin{proposition}\label{thm:factor-minors}
Let \(\F\) be a field, \(s\ge2\), and \(M_1,\ldots,M_s\) be binary
matroids representable over \(\F\).  If \(N_i\) is
isomorphic to a minor of \(M_i\) for each \(i\in[s]\), then
\(T_{\F}(N_1,\ldots,N_s)\) is isomorphic to a minor of
\(T_{\F}(M_1,\ldots,M_s)\).
\end{proposition}
\begin{proof}
Let \(A=[\boldsymbol a_e:e\in E]\) and
\(B=[\boldsymbol b_f:f\in F]\) be arbitrary full row rank matrices
over \(\F\).  We first prove
the assertion for a single deletion or contraction in the first factor.  Let \(e\in E\).  Deleting the fibre \(\{e\}\times F\) simply deletes from
\(A\otimes B\) the columns whose first label is \(e\).  Hence
\[
M[A\otimes B]\backslash(\{e\}\times F)=M[A_{E\setminus\{e\}}\otimes B].
\]

We next contract the fibre.  If \(e\) is a loop of \(M[A]\), then
\(\boldsymbol{a}_e=\boldsymbol{0}\), so every element of \(\{e\}\times F\)
is a loop of \(M[A\otimes B]\), and contraction agrees with deletion.
Suppose  \(e\) is a nonloop.  Set \(E_0=E\setminus\{e\}\).  Apply
elementary row operations so that the column labelled by \(e\) is
\(\boldsymbol{e}_1\).  Then
\[
A=
 \left[
 \begin{array}{c|c}
   \boldsymbol{\alpha}&1\\ \hline
   A_0&\boldsymbol{0}
 \end{array}
 \right],
 \qquad
 \text{with column blocks }E_0\mid\{e\}.
\]
Here \(\boldsymbol{\alpha}\) is a row vector, and the standard matrix rule for
contraction \cite[Proposition~3.2.6]{oxley} identifies \(A_0\) as a
representation of \(M[A]/e\) on \(E_0\).  Let \(r=\rank(B)\).
If \(r=0\), every tensor column is zero, so contraction of the fibre
agrees with deletion and leaves the matroid \(M[A_0\otimes B]\).
Assume \(r>0\), and choose a basis \(J\) of \(M[B]\).
Since \(B\) has full row rank, row operations give \(B_J=I_r\).
Put \(B_0=B_{F\setminus J}\).  Grouping tensor rows by the rows of
\(A\) and columns into \(E_0\times F\), \(\{e\}\times J\), and
\(\{e\}\times(F\setminus J)\), we obtain
\[
A\otimes B=
 \left[
 \begin{array}{c|c|c}
   \boldsymbol{\alpha}\otimes B&I_r&B_0\\ \hline
   A_0\otimes B&0&0
 \end{array}
 \right].
\]
The columns labelled by \(\{e\}\times J\) are the first \(r\) standard
basis vectors.  Contracting them deletes these columns and their
pivot rows \cite[Proposition~3.2.6]{oxley}, leaving \([\,A_0\otimes B\mid 0\,]\).
The columns labelled by \(\{e\}\times(F\setminus J)\) are therefore
loops, whose contraction agrees with deletion.  Consequently,
\[
M[A\otimes B]/(\{e\}\times F)=M[A_0\otimes B].
\]
Thus contracting a fibre induces contraction of its label in the
first factor.  Both matrix identities hold for arbitrary representing
matrices \(A\) and \(B\).

For the stated product, fix a coordinate \(i\), permute it to the
first position, and put \(A=A_i\) and
\(B=\bigotimes_{j\ne i}A_j\), choosing full row rank factor matrices.
The rank formula for Kronecker products shows that \(B\) has full
row rank, including the rank-zero case.  The preceding identities
lift every deletion and contraction in \(M_i\) to a minor operation
in the full product.  Repeat in each coordinate.  The induced factor
matrices represent the corresponding binary minors, and
Proposition~\ref{thm:basic-represented-product} permits their
replacement by any other representations of those minors.
Relabelling yields the asserted isomorphism.
\end{proof}

If every factor has a nonloop, fixing one nonloop in every
coordinate except \(i\) identifies \(M_i\) with a restriction of
the full product: tensoring with a fixed nonzero vector is injective.
A regular, graphic, or cographic product therefore has factors in the
corresponding minor-closed class.  The nonloop hypothesis is necessary,
since a rank-zero factor makes every tensor column zero.

\begin{proposition}\label{prop:tensor-direct-sums}
For a field \(\F\), binary matroids \(M,N_1,N_2\) representable over
\(\F\), and an integer \(q\ge0\),
\begin{equation}\label{eq:free-factor}
 T_{\F}(M,U_{q,q})\cong M^{\oplus q},
\end{equation}
and, under the natural ground-set identification,
\[
T_{\F}(M,N_1\oplus N_2)
 =T_{\F}(M,N_1)\oplus T_{\F}(M,N_2).
\]
The same identities hold in any coordinate of a higher product.
\end{proposition}
\begin{proof}
Use the identity matrix \(I_q\) to represent \(U_{q,q}\) and a
block diagonal matrix to represent \(N_1\oplus N_2\).
After permuting rows and labelled columns, the Kronecker product
is block diagonal with precisely the asserted summands.  The same
argument applies with the remaining coordinates grouped into one
matrix.  Representation invariance completes the proof.
\end{proof}

The corresponding abstract statements appear in
\cite[Proposition~2.5 and Lemma~2.3]{anderson-matroid-products}.
The direct sum identity specializes to the usual block decomposition
for graphs.
\begin{proposition}\label{prop:block-decomposition}
Let \(G,H\) be simple graphs with nonempty edge sets, and let \(\F\)
be a field.  Let \(G_1,\ldots,G_s\) and \(H_1,\ldots,H_t\) be the
blocks of \(G\) and \(H\), respectively, with nonempty edge sets.  Then
\[
 M_{\F}(G,H)=\bigoplus_{i=1}^{s}\bigoplus_{j=1}^{t}
 M_{\F}(G_i,H_j).
\]
\end{proposition}
\begin{proof}
Every circuit is contained in a single block, so
\cite[Section~4.1 and Corollary~4.2.9]{oxley} gives
\[
 M(G)=\bigoplus_{i=1}^{s}M(G_i),\qquad
 M(H)=\bigoplus_{j=1}^{t}M(H_j).
\]
Applying Proposition~\ref{prop:tensor-direct-sums} in both coordinates
proves the asserted decomposition.  If either \(G_i\) or \(H_j\) is a bridge block,
then \eqref{eq:free-factor} identifies the corresponding summand with
the cycle matroid of the other block.
\end{proof}

We next remove loops and parallel elements uniformly for all numbers
of factors.  The abstract two-factor versions are
\cite[Lemma~2.4]{anderson-matroid-products} and
\cite[Lemma~4.4 of the expanded version]{berczi-skew-tensor}.
For each \(M_i\), choose a set \(S_i\) containing exactly one element
from each nonloop parallel class.  Write
\(\operatorname{si}(M_i)=M_i|S_i\), and let \(\rho_i\) send a
nonloop to its representative in \(S_i\).

\begin{proposition}\label{prop:regular-simplification}
Let \(\F\) be a field, \(s\ge2\), and \(M_1,\ldots,M_s\) be finite
binary matroids representable over \(\F\).  Then
\[
 T_{\F}\bigl(\operatorname{si}(M_1),\ldots,\operatorname{si}(M_s)\bigr)
 =T_{\F}(M_1,\ldots,M_s)\bigm|(S_1\times\cdots\times S_s).
\]
The loops of the full product are precisely the tuples with a loop
in at least one coordinate.  Its nonloop parallel classes are the
fibres of \(\rho_1\times\cdots\times\rho_s\).
If every \(S_i\) is nonempty, the full product is graphic,
cographic, regular, or binary exactly when the product of the
simplifications belongs to that class.  If some \(S_i\) is empty,
the full product is an all-loop matroid, possibly empty, and belongs
to all four classes.
\end{proposition}
\begin{proof}
Restricting the tensor matrix to \(S_1\times\cdots\times S_s\)
proves the displayed identity.  A pure tensor of vectors vanishes
exactly when one factor vanishes.  Two nonzero pure tensors are
proportional exactly when their factor vectors are proportional in
every coordinate.  For the nontrivial direction, apply linear
functionals on all but one coordinate, choosing them to be nonzero
on the vectors of one tensor; the equality ensures nonzero values
on the other tensor as well.  The remaining vectors are therefore
proportional.  Applying this argument in each coordinate proves
the description of the parallel classes.

Thus the full product is obtained from the displayed restriction
by adjoining loops and parallel elements.  All four classes are
closed under these operations.  In a graphic representation, add
loops or parallel edges.  For cographic matroids, duality converts
the operations into coloop additions and series extensions, which
preserve graphicity \cite[Section~5.4]{oxley}.
For binary and regular matroids, adjoining a zero or duplicate column
preserves representability over \(\operatorname{GF}(2)\) and total
unimodularity, respectively.
Conversely, all four classes are minor-closed.  If some \(S_i\)
is empty, every tensor column is zero and the stated degenerate
case follows.
\end{proof}

\section{Reduction to graphic factors and constructions}\label{sec:positive-constructions}
This section develops the tools for the classification: the minor
containments in Proposition~\ref{prop:specified-minors}, the reduction
to graphic factors in Theorem~\ref{thm:graphic-reduction}, and the
cographic and totally unimodular constructions in
Propositions~\ref{prop:outerplanar-cographic} and
\ref{prop:triangle-series-parallel}, respectively.

\subsection{Reduction to graphic factors}\label{subsec:graphic-reduction}
We collect the minor containments used in the reduction and in the
classifications of two-factor and higher products.  Their proofs are
given in Section~\ref{sec:minor-proofs}.
\begin{proposition}\label{prop:specified-minors}
Let \(\F\) be a field.
\begin{enumerate}[label=\textup{(\roman*)},leftmargin=2.5em]
\item For each pair
\[
 (G,H)\in\{(C_3,K_4),\ (K_4-e,K_4-e),\ (C_4,K_{2,3})\},
\]
where \(e\) is an edge of \(K_4\), the matroid \(M_{\F}(G,H)\)
has a \(U_{2,4}\)-minor if \(\operatorname{char}(\F)\ne2\), and an
\(F_7\)-minor if \(\operatorname{char}(\F)=2\).  In particular, each
of these three products is nonregular over every field.
\item The matroid \(M_{\F}(C_3,K_{2,3})\) has an
\(M(K_{3,3})\)-minor and is therefore not cographic.
\item The matroid \(T_{\F}(U_{2,3},U_{2,3},U_{2,3})\) has a
\(U_{2,4}\)-minor if \(\operatorname{char}(\F)\ne2\), and an
\(F_7\)-minor if \(\operatorname{char}(\F)=2\).  In particular,
it is nonregular over every field.
\end{enumerate}
\end{proposition}

In characteristic two, the product in part~\textup{(iii)} is dual
to the nonregular simplicial matroid with three parts of size three
studied by Lindstr\"om \cite{lindstrom-nonregular}, as recorded in
\cite[Exercise~6.6, p.~112]{cordovil-lindstrom}.
Here we give explicit excluded minors over every field.

\begin{remark}\label{ex:regular-characteristic}
The two matroids
\[
 T_{\operatorname{GF}(2)}(U_{2,3},M(K_4))
 \quad\text{and}\quad
 T_{\QQ}(U_{2,3},M(K_4))
\]
are abstract tensor products of the same regular factors.
The first is binary by Lemma~\ref{lem:prime-field-reduction}; the
second has a \(U_{2,4}\)-minor by
Proposition~\ref{prop:specified-minors}\textup{(i)} and is nonbinary.
They are therefore nonisomorphic.  This example shows both that the
represented product can depend on the characteristic and that
abstract tensor products of regular matroids need not be unique.
\end{remark}

The intrinsic series--parallel and outerplanar conditions have the
following graph interpretations.
\begin{lemma}\label{lem:regular-sp-outerplanar}
Every series--parallel simple binary matroid is graphic, and it has a
representation by a simple graph whose cyclic blocks are
series--parallel networks.  An outerplanar simple binary matroid has a
simple outerplanar graph representation.  Conversely, the cycle matroid
of a simple outerplanar graph has neither an \(M(K_4)\)- nor an
\(M(K_{2,3})\)-minor.
For a simple graph \(G\), \(M(G)\) is a cactus matroid if and only if
\(G\) is a cactus.  It is a triangular cactus matroid if and only if
every cyclic block of \(G\) is a triangle.
\end{lemma}
\begin{proof}
Let \(M\) be a simple binary matroid with no \(M(K_4)\)-minor.
Since \(M\) has neither a \(U_{2,4}\)- nor an \(M(K_4)\)-minor,
the series--parallel characterization
\cite{brylawski-series-parallel,bonin-long} implies that each
connected component other than a coloop is the cycle matroid of
a series--parallel network.  Choose a simple graph representation;
parallel elements are absent and the component decomposition agrees
with the block decomposition.  If \(M\) also avoids \(M(K_{2,3})\),
this graph cannot contain \(K_4\) or \(K_{2,3}\) as a graph minor, by
Fact~\ref{fact:graph-minors}; it is outerplanar by
Fact~\ref{fact:outerplanar-minors}.

Conversely, a minor of the cycle matroid of an outerplanar graph has
an outerplanar graph representation, since outerplanarity is closed
under graph deletion and contraction.  Neither \(M(K_4)\) nor
\(M(K_{2,3})\) has such a representation.  Indeed, after isolated
vertices are discarded, any graph representing either matroid is
simple and \(2\)-connected.  A graph representing \(M(K_4)\) has
four vertices and six edges, and hence is \(K_4\).  A graph
representing \(M(K_{2,3})\) has five vertices and six edges, so it
is a theta graph: three internally disjoint paths with the same
endpoints.  Its three cycles all have length four, so all three
paths have length two and the graph is \(K_{2,3}\).
Finally, connected components of \(M(G)\) correspond to blocks
of \(G\) with edges.  Such a component is a circuit matroid
exactly when the block is a cycle, and a single coloop exactly
when the block is a bridge.
\end{proof}

Every simple nonfree matroid has a circuit of size at least three.
Restricting to that circuit and contracting all but three elements
gives a \(U_{2,3}\)-minor.  This observation and
Proposition~\ref{prop:specified-minors}\textup{(i)} yield the main
reduction.
\begin{theorem}\label{thm:graphic-reduction}
Let \(M,N\) be simple nonfree binary matroids representable over
a common field \(\F\).
If \(T_{\F}(M,N)\) is regular, then each factor is a series--parallel
graphic matroid.
\end{theorem}
\begin{proof}
By the preceding observation, \(N\) has a \(U_{2,3}=M(C_3)\)-minor.
If \(M\) had an \(M(K_4)\)-minor, Proposition~\ref{thm:factor-minors}
would give \(T_{\F}(M(K_4),U_{2,3})\) as a minor of
\(T_{\F}(M,N)\).  This minor is nonregular over every field by
Proposition~\ref{prop:specified-minors}\textup{(i)}, contradicting minor
closure.  The same argument excludes \(M(K_4)\) from \(N\).
Apply Lemma~\ref{lem:regular-sp-outerplanar} to both factors.
\end{proof}

Graphic and cographic matroids are regular, so the same reduction
applies to either of those class-membership questions when both
simple factors are nonfree.  In characteristic different from two,
Fact~\ref{fact:matroid-classes} also reduces binarity to regularity.
In characteristic two, the product is binary by
Lemma~\ref{lem:prime-field-reduction}.  Thus, for two simple nonfree
factors, it remains to classify products of graphic factors and
translate the answer back to binary factors over a common
representation field.

\subsection{Cographic constructions}\label{subsec:cographic-constructions}
For a graph \(G\), fix an orientation \(o(G)\).  Its incidence
matrix \(A_{o(G)}\) over \(\F\) has rows indexed by \(V(G)\) and
columns by \(E(G)\): an edge directed from \(u\) to \(v\) has entry
\(1\) at \(u\), \(-1\) at \(v\), and zero elsewhere; a loop gives a
zero column.  Over every field,
\(M[A_{o(G)}]=M(G)\) and \(\rank(A_{o(G)})=r(G)\)
\cite[Lemma~5.1.3]{oxley}.  A \emph{full row rank incidence
representation} \(A_G\) has \(r(G)\) rows and
\(\row(A_G)=\row(A_{o(G)})\).

\Needspace{5\baselineskip}
A \emph{reduced incidence matrix} is obtained by deleting one vertex
row per component, including each isolated vertex.  The vertex rows
in each component sum to zero, so the retained rows span the same
row space.  There are \(r(G)\) retained rows, hence the reduced
matrix has full row rank.  These representations of \(M(G)\) may be used to
form \(M_{\F}(G,H)\) by
Proposition~\ref{thm:basic-represented-product}.

For a simple graph \(G\), write
\[
 Z_{\F}(G):=\ker(A_{o(G)}),\qquad
 B_{\F}(G):=\row(A_{o(G)})
\]
for its \(\F\)-cycle and cut spaces; compare
\cite[Chapters~4--5]{biggs-algebraic-graph-theory}.
For \(v\in V(G)\), let \(\partial_G(v)\) be the set of edges
incident with \(v\), and let \(\boldsymbol b_G(v)\) be the row of
\(A_{o(G)}\) indexed by \(v\).  For a cycle \(\Gamma\)
with a chosen cyclic orientation, let \(\boldsymbol z_\Gamma\) be
its edge vector, equal to \(1\) or \(-1\) on \(\Gamma\) according
as its orientation agrees or disagrees with \(o(G)\), and zero
off \(\Gamma\).  The incidence rows span \(B_{\F}(G)\), the cycle
vectors span \(Z_{\F}(G)\), and
\(B_{\F}(G)=Z_{\F}(G)^\perp\), where orthogonality is taken with
respect to the standard bilinear coordinate pairing on \(\F^{E(G)}\).
Thus the dual-representation
theorem \cite[Theorem~2.2.8 and Proposition~2.2.23]{oxley} gives
\begin{equation}\label{eq:cycle-space-dual-representation}
 A\text{ has full row rank and }\row(A)=Z_{\F}(G)
 \quad\Longrightarrow\quad
 M[A]\cong M^*(G).
\end{equation}

For \(k\ge3\), orient \(C_k\) cyclically and label its edges by
\([k]\).  Row operations on a reduced incidence matrix give
\begin{equation}\label{eq:cycle-incidence-matrix}
 \Delta_k=[I_{k-1}\mid-\boldsymbol{1}_{k-1}].
\end{equation}
Unless otherwise stated, we use this full row rank representation
of \(M(C_k)\) for cycle factors.  For full row rank incidence representations \(A_G\)
and \(A_H\) of simple graphs \(G,H\),
Proposition~\ref{thm:basic-represented-product} and
\[
 \row(A_G\otimes A_H)=B_{\F}(G)\otimes B_{\F}(H)
\]
show that cographicity of \(M_{\F}(C_m,G)\) follows from a graph
\(X_{m,G}\) whose labelled cycle space is
\(B_{\F}(C_m)\otimes B_{\F}(G)\).

For a plane graph, the unbounded face is the \emph{outer face}; the
other faces are \emph{bounded}.  In a \(2\)-connected simple plane
graph, the outer-face boundary is a cycle.  Its edges are the
\emph{outer edges}; the other edges are \emph{internal}.  The
\emph{geometric dual} \(G^*\) has a vertex for each face and a dual
edge \(e^*\) joining the faces on the two sides of \(e\), with a
loop if those faces coincide.  Under \(e\leftrightarrow e^*\),
\(M(G^*)\cong M^*(G)\), and dual edges of a bond form a cycle of
\(G^*\) \cite[Proposition~5.2.1]{oxley}.  The \emph{weak dual} of an
outerplane graph is obtained by deleting the outer-face vertex from
\(G^*\).

Let \(m\ge3\) and let \(G\) be a \(2\)-connected outerplanar
simple graph.  Fix an outerplane embedding and an orientation
\(o(G)\); orient \(C_m\) cyclically and label its edges by \([m]\).
Write \(E_\partial(G)\) for the outer edges and \(\mathcal F(G)\)
for the bounded faces.  The outer boundary is a Hamilton cycle
\cite[p.~436]{chartrand-harary}, so every vertex of \(G\) meets
exactly two outer edges.  Construct \(X_{m,G}\) as follows.
\begin{itemize}
\item \textbf{Weak dual.} Let \(v_\infty\) be the outer-face vertex
of \(G^*\).  Label each dual edge \(e^*\) by \(e\) and orient it
from the face on the left of \(e\) to the face on its right.  The
weak dual \(T_G:=G^*-v_\infty\) is a tree on \(\mathcal F(G)\)
\cite[Theorem~1]{fleischner-geller-harary}.
\item \textbf{Split weak dual.} For each \(e\in E_\partial(G)\),
attach a new leaf \(x_e\) to the bounded face incident with \(e\),
using an edge labelled by \(e\) with the inherited dual orientation.
The resulting tree \(\widehat T_G\) has edges labelled by \(E(G)\).
\item \textbf{Layer graph.} Take \(m\) copies of \(\widehat T_G\),
indexed by \([m]\), and identify the copies of each \(x_e\).
Call the resulting graph \(X_{m,G}\); label the copy of an edge \(e\)
in layer \(i\) by \((i,e)\), retaining its orientation.  Then
\begin{equation}\label{eq:vertex-edge}
V(X_{m,G})=([m]\times\mathcal F(G))\sqcup\{x_e:e\in E_\partial(G)\},\qquad E(X_{m,G})=[m]\times E(G).
\end{equation}
Every layer is connected, and the vertices \(x_e\) are shared by
all layers, so \(X_{m,G}\) is connected.
\end{itemize}
Figure~\ref{fig:cographic-construction} illustrates the construction for a square with a diagonal and \(C_3\).

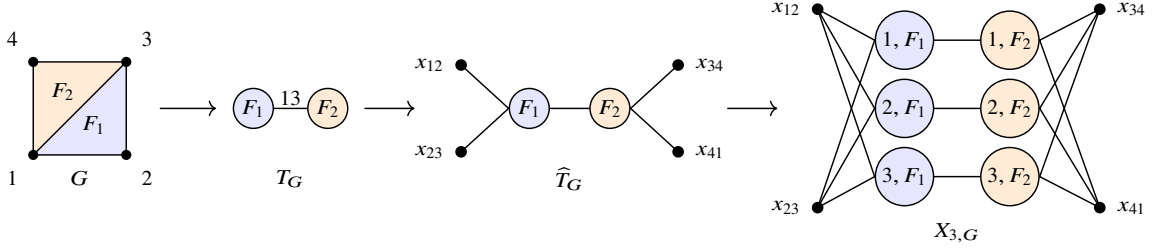
\begin{figure}[H]
\centering
\begin{tikzpicture}[scale=.82,line width=.55pt,line cap=round,line join=round]
  \fill[blue!10] (0,0)--(1.5,0)--(1.5,1.5)--cycle;
  \fill[orange!16] (0,0)--(1.5,1.5)--(0,1.5)--cycle;
  \node[graphvertex,label=below left:{\scriptsize \(1\)}] (v0) at (0,0) {};
  \node[graphvertex,label=below right:{\scriptsize \(2\)}] (v1) at (1.5,0) {};
  \node[graphvertex,label=above right:{\scriptsize \(3\)}] (v2) at (1.5,1.5) {};
  \node[graphvertex,label=above left:{\scriptsize \(4\)}] (v3) at (0,1.5) {};
  \draw (v0.center)--(v1.center)--(v2.center)--(v3.center)--(v0.center)--(v2.center);
  \node[graphlabel] at (.98,.48) {\(F_1\)};
  \node[graphlabel] at (.48,1.02) {\(F_2\)};
  \node[graphlabel] at (.75,-.38) {\(G\)};

  \draw[->] (2.05,.75)--(2.95,.75);

  \node[facevertex,fill=blue!10] (t1) at (3.55,.75) {\(F_1\)};
  \node[facevertex,fill=orange!16] (t2) at (4.75,.75) {\(F_2\)};
  \draw (t1.east)--node[above,graphlabel]{\(13\)}(t2.west);
  \node[graphlabel] at (4.15,-.38) {\(T_G\)};

  \draw[->] (5.35,.75)--(6.15,.75);

  \node[leafvertex,label=left:{\scriptsize \(x_{12}\)}] (a1) at (6.90,1.45) {};
  \node[leafvertex,label=left:{\scriptsize \(x_{23}\)}] (a2) at (6.90,.05) {};
  \node[facevertex,fill=blue!10] (f1) at (8.00,.75) {\(F_1\)};
  \node[facevertex,fill=orange!16] (f2) at (9.30,.75) {\(F_2\)};
  \node[leafvertex,label=right:{\scriptsize \(x_{34}\)}] (a3) at (10.40,1.45) {};
  \node[leafvertex,label=right:{\scriptsize \(x_{41}\)}] (a4) at (10.40,.05) {};
  \draw (a1.center)--(f1.west) (a2.center)--(f1.west)
        (f1.east)--(f2.west)
        (f2.east)--(a3.center) (f2.east)--(a4.center);
  \node[graphlabel] at (8.65,-.38) {\(\widehat T_G\)};

  \draw[->] (11.20,.75)--(12.00,.75);

  \node[leafvertex,label=left:{\scriptsize \(x_{12}\)}] (b1) at (12.65,2.35) {};
  \node[leafvertex,label=left:{\scriptsize \(x_{23}\)}] (b2) at (12.65,-.85) {};
  \node[leafvertex,label=right:{\scriptsize \(x_{34}\)}] (b3) at (17.20,2.35) {};
  \node[leafvertex,label=right:{\scriptsize \(x_{41}\)}] (b4) at (17.20,-.85) {};
  \foreach \i/\y in {1/1.85,2/.75,3/-.35}{
    \node[facevertex,fill=blue!10] (g\i) at (14.05,\y) {\(\i,F_1\)};
    \node[facevertex,fill=orange!16] (h\i) at (15.75,\y) {\(\i,F_2\)};
    \draw (b1.center)--(g\i.west) (b2.center)--(g\i.west)
          (g\i.east)--(h\i.west)
          (h\i.east)--(b3.center) (h\i.east)--(b4.center);
  }
  \node[graphlabel] at (14.90,-1.25) {\(X_{3,G}\)};
\end{tikzpicture}
\caption{The construction
\(G\longmapsto T_G\longmapsto\widehat T_G\longmapsto X_{3,G}\) for a
square with one diagonal.}
\label{fig:cographic-construction}
\end{figure}

For two cycle factors, a cographic realization in characteristic zero
is recorded in \cite[Section~5.2]{hidaka-itoh}.  The following
proposition treats a cycle paired with a \(2\)-connected outerplanar
graph over an arbitrary field.
\begin{proposition}
\label{prop:outerplanar-cographic}
Let \(m\ge3\), \(\F\) be a field, and \(G\) be a \(2\)-connected outerplanar simple graph.  Then the graph \(X_{m,G}\) constructed above satisfies
\[
M_{\F}(C_m,G)\cong M^*(X_{m,G}).
\]
Moreover, if \(G=C_n\), where \(n\ge3\), then
\begin{equation}\label{eq:two-cycle-identity}
 M_{\F}(C_m,C_n)\cong M^*(K_{m,n}).
\end{equation}
\end{proposition}
\begin{proof}
Set \(A_{C_m}=\Delta_m\), whose \(i\)-th row is
\(\boldsymbol e_i^{(m)}-\boldsymbol e_m^{(m)}\), and let \(A_G\)
be the reduced incidence matrix obtained by deleting the row at
\(v_0\in V(G)\).  These are full row rank representations of the
factors; the rows of \(A_G\) are
\(\boldsymbol b_G(v)\) for \(v\ne v_0\).

For each such \(v\), the graph \(G-v\) is connected, so
\(\partial_G(v)\) is a bond.  Its dual edges form a cycle
\(C_v^*\) of \(G^*\).  The two outer edges incident with \(v\)
correspond to the two edges of \(C_v^*\) incident with
\(v_\infty\).  Splitting \(v_\infty\) therefore turns \(C_v^*\)
into a path \(P_v\) between the corresponding leaves of
\(\widehat T_G\).

Traverse \(C_v^*\) with the face of \(G^*\) corresponding to
\(v\) on the right, and give \(P_v\) the induced orientation.
At an edge directed away from \(v\), this traversal agrees with
the dual-edge orientation; at an edge directed towards \(v\),
it is opposite.  Thus the signed edge vector of \(P_v\) is
\(\boldsymbol b_G(v)\).  Let \(P_v^j\) be its copy in layer
\(j\in[m]\).  For \(i\in[m-1]\), the paths \(P_v^i\) and
\(P_v^m\) have the same endpoints and disjoint interiors, so
their union is a cycle \(\Gamma_{i,v}\).  Traverse \(P_v^i\)
forward and \(P_v^m\) backward.  The resulting cycle vector is
\[
 \boldsymbol{z}_{\Gamma_{i,v}}=(\boldsymbol{e}_i-\boldsymbol{e}_m) \otimes\boldsymbol b_G(v).
\]
Figure~\ref{fig:dual-path-cycle} illustrates the construction of these cycle vectors.
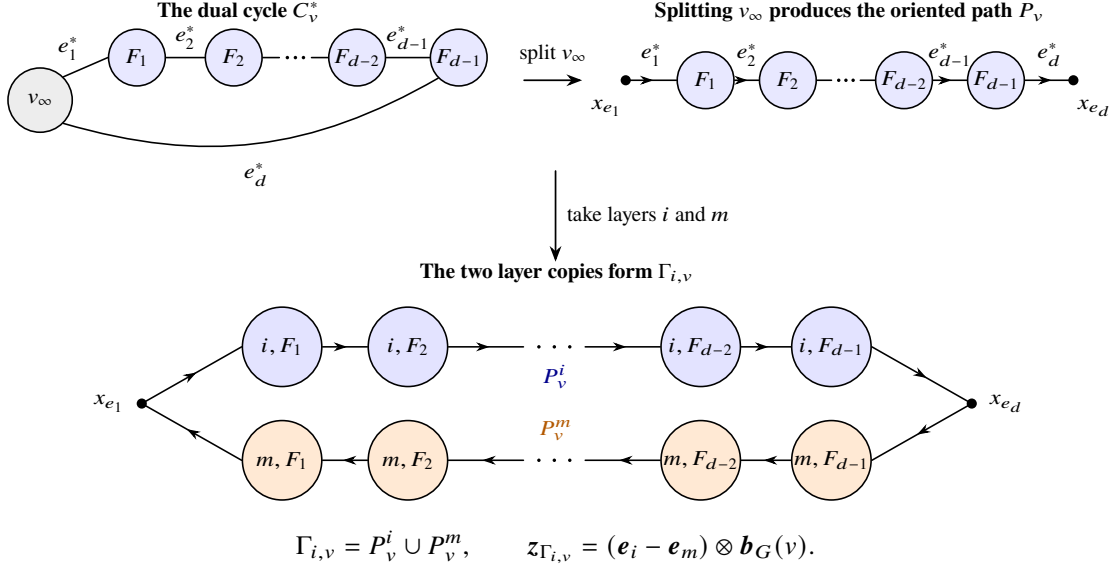
\begin{figure}[!ht]
\centering
\begin{tikzpicture}[
  x=.64cm,y=.75cm,
  line cap=round,
  line join=round,
  line width=.60pt,
  >=Stealth,
  localface/.style={circle,draw=black,fill=blue!9,minimum size=7.5mm,inner sep=1pt,font=\scriptsize},
  localouter/.style={circle,draw=black,fill=black!7,minimum size=8.5mm,inner sep=1pt,font=\scriptsize},
  localleaf/.style={circle,fill=black,inner sep=1.35pt},
  ilayer/.style={circle,draw=black,fill=blue!11,minimum size=10.5mm,inner sep=0pt,font=\scriptsize},
  mlayer/.style={circle,draw=black,fill=orange!17,minimum size=10.5mm,inner sep=0pt,font=\scriptsize},
  locallab/.style={font=\scriptsize,inner sep=1pt},
  localtitle/.style={font=\scriptsize\bfseries},
  patharrow/.style={
    postaction={decorate},
    decoration={markings,mark=at position .55 with
      {\arrow{Stealth[length=1.8mm,width=1.2mm]}}}
  }
]
  \node[localtitle,anchor=center] at (4.10,1.55)
    {The dual cycle \(C_v^*\)};
  \node[localouter] (vinf) at (0,0) {\(v_\infty\)};
  \node[localface] (c1) at (2.0,.75) {\(F_1\)};
  \node[localface] (c2) at (4.0,.75) {\(F_2\)};
  \node[localface] (cdm2) at (6.55,.75) {\(F_{d-2}\)};
  \node[localface] (cdm1) at (8.65,.75) {\(F_{d-1}\)};
  \draw (vinf.north east)--node[locallab,above left=2pt]{\(e_1^*\)}(c1.west);
  \draw (c1.east)--node[locallab,above=2pt]{\(e_2^*\)}(c2.west);
  \coordinate (cLeft) at ($(c2.east)+(.34,0)$);
  \coordinate (cRight) at ($(cdm2.west)+(-.34,0)$);
  \draw (c2.east)--(cLeft);
  \foreach \t in {.25,.50,.75}
    \fill ($(cLeft)!\t!(cRight)$) circle[radius=.65pt];
  \draw (cRight)--(cdm2.west);
  \draw (cdm2.east)--node[locallab,above=2pt]{\(e_{d-1}^*\)}(cdm1.west);
  \draw (cdm1.south west) to[bend left=23]
        node[locallab,below=6pt]{\(e_d^*\)}(vinf.south east);

  \draw[-{Stealth},line width=.65pt] (10.0,.35)--(11.2,.35)
        node[midway,above=5pt,locallab]{split \(v_\infty\)};

  \node[localtitle,anchor=center] at (16.73,1.55)
    {Splitting \(v_\infty\) produces the oriented path \(P_v\)};
  \node[localleaf] (x1) at (12.1,.35) {};
  \node[localface] (p1) at (13.75,.35) {\(F_1\)};
  \node[localface] (p2) at (15.45,.35) {\(F_2\)};
  \node[localface] (pdm2) at (17.85,.35) {\(F_{d-2}\)};
  \node[localface] (pdm1) at (19.75,.35) {\(F_{d-1}\)};
  \node[localleaf] (xd) at (21.35,.35) {};
  \node[locallab,anchor=north east] at ($(x1)+(0,-.28)$) {\(x_{e_1}\)};
  \node[locallab,anchor=north west] at ($(xd)+(0,-.28)$) {\(x_{e_d}\)};
  \draw[patharrow] (x1.center)--node[locallab,above=4pt]{\(e_1^*\)}(p1.west);
  \draw[patharrow] (p1.east)--node[locallab,above=4pt]{\(e_2^*\)}(p2.west);
  \coordinate (pLeft) at ($(p2.east)+(.30,0)$);
  \coordinate (pRight) at ($(pdm2.west)+(-.30,0)$);
  \draw (p2.east)--(pLeft);
  \foreach \t in {.25,.50,.75}
    \fill ($(pLeft)!\t!(pRight)$) circle[radius=.65pt];
  \draw (pRight)--(pdm2.west);
  \draw[patharrow] (pdm2.east)--node[locallab,above=4pt]{\(e_{d-1}^*\)}(pdm1.west);
  \draw[patharrow] (pdm1.east)--node[locallab,above=4pt]{\(e_d^*\)}(xd.center);

  \draw[-{Stealth},line width=.65pt] (10.65,-1.25)--(10.65,-2.85)
        node[midway,right=3pt,locallab]{take layers \(i\) and \(m\)};

  \node[localtitle,anchor=center] at (10.68,-3.05)
    {The two layer copies form \(\Gamma_{i,v}\)};
  \node[localleaf,label=left:{\scriptsize \(x_{e_1}\)}] (y1) at (2.1,-5.35) {};
  \node[localleaf,label=right:{\scriptsize \(x_{e_d}\)}] (yd) at (19.25,-5.35) {};
  \node[ilayer] (i1) at (5.0,-4.35) {\(i,F_1\)};
  \node[ilayer] (i2) at (7.6,-4.35) {\(i,F_2\)};
  \node[ilayer] (idm2) at (13.65,-4.35) {\(i,F_{d-2}\)};
  \node[ilayer] (idm1) at (16.35,-4.35) {\(i,F_{d-1}\)};
  \node[mlayer] (m1) at (5.0,-6.35) {\(m,F_1\)};
  \node[mlayer] (m2) at (7.6,-6.35) {\(m,F_2\)};
  \node[mlayer] (mdm2) at (13.65,-6.35) {\(m,F_{d-2}\)};
  \node[mlayer] (mdm1) at (16.35,-6.35) {\(m,F_{d-1}\)};

  \draw[patharrow] (y1.center)--(i1.west);
  \draw[patharrow] (i1.east)--(i2.west);
  \coordinate (iLeft) at ($(i2.east)+(1.55,0)$);
  \coordinate (iRight) at ($(idm2.west)+(-1.55,0)$);
  \draw[patharrow] (i2.east)--(iLeft);
  \foreach \t in {.25,.50,.75}
    \fill ($(iLeft)!\t!(iRight)$) circle[radius=.65pt];
  \draw[patharrow] (iRight)--(idm2.west);
  \draw[patharrow] (idm2.east)--(idm1.west);
  \draw[patharrow] (idm1.east)--(yd.center);
  \draw[patharrow] (yd.center)--(mdm1.east);
  \draw[patharrow] (mdm1.west)--(mdm2.east);
  \coordinate (mLeft) at ($(m2.east)+(1.55,0)$);
  \coordinate (mRight) at ($(mdm2.west)+(-1.55,0)$);
  \draw[patharrow] (mdm2.west)--(mRight);
  \foreach \t in {.25,.50,.75}
    \fill ($(mLeft)!\t!(mRight)$) circle[radius=.65pt];
  \draw[patharrow] (mLeft)--(m2.east);
  \draw[patharrow] (m2.west)--(m1.east);
  \draw[patharrow] (m1.west)--(y1.center);
  \node[locallab,text=blue!55!black] at (10.65,-4.88) {\(P_v^i\)};
  \node[locallab,text=orange!70!black] at (10.65,-5.82) {\(P_v^m\)};
  \node[font=\footnotesize] at (10.65,-7.88)
    {\(\displaystyle
      \Gamma_{i,v}=P_v^i\cup P_v^m,
      \qquad
      \boldsymbol{z}_{\Gamma_{i,v}}
      =(\boldsymbol{e}_i-\boldsymbol{e}_m)\otimes
        \boldsymbol b_G(v).\)};
\end{tikzpicture}
\caption{Splitting \(v_\infty\) turns \(C_v^*\) into \(P_v\);
the oppositely traversed copies in layers \(i\) and \(m\) form
\(\Gamma_{i,v}\).  Here \(d=\deg_G(v)\), and the incident edges
\(e_1,\ldots,e_d\) are indexed in the order in which their dual
edges are traversed along \(C_v^*\), starting at \(v_\infty\).
The edges \(e_1\) and \(e_d\) are the two outer edges.}
\label{fig:dual-path-cycle}
\end{figure}
The vectors \(\boldsymbol z_{\Gamma_{i,v}}\) are precisely the rows of \(A_{C_m}\otimes A_G\).  Consequently,
\begin{equation}\label{eq:outerplanar-rowspace-inclusion}
 \row(A_{C_m}\otimes A_G)\subseteq Z_{\F}(X_{m,G}).
\end{equation}

By Euler's formula and \eqref{eq:vertex-edge},
\[
 \begin{aligned}
 \dim Z_{\F}(X_{m,G})
 &=m|E(G)|-|V(X_{m,G})|+1\\
 &=(m-1)(|V(G)|-1)
 =\rank(A_{C_m}\otimes A_G).
 \end{aligned}
\]
Thus \eqref{eq:outerplanar-rowspace-inclusion} is an equality, and
\eqref{eq:cycle-space-dual-representation} gives
\(M_{\F}(C_m,G)\cong M^*(X_{m,G})\).

If \(G=C_n\), then \(\widehat T_G\cong K_{1,n}\).
Identifying equally labelled leaves in \(m\) copies gives
\(X_{m,C_n}\cong K_{m,n}\), proving
\eqref{eq:two-cycle-identity}.
\end{proof}

\subsection{A regular construction for series--parallel graphs}
For a triangle factor, regularity extends to series--parallel
networks.  Using their construction by edge subdivisions and
parallel-edge additions \cite{duffin-series-parallel,brylawski-series-parallel},
we construct an integral full row rank incidence representation
\(A\) for which
\(\Delta_3\otimes A\) is totally unimodular.  The next two lemmas
show that these operations preserve the required property.

\begin{lemma}\label{lem:triangle-parallel-tu}
Let \(G\) be a loopless graph and let \(G^+\) be obtained by adding
an edge \(f\) parallel to \(e\in E(G)\).  Suppose that
\(A=[\boldsymbol a_g:g\in E(G)]\) is an integral full row rank
incidence representation of \(M(G)\) over \(\QQ\).  If
\(\Delta_3\otimes A\) is totally unimodular, then
\[
 A^+=[A\mid\boldsymbol a_e],
\]
with its last column labelled by \(f\), is an integral full row rank
incidence representation of \(M(G^+)\) over \(\QQ\), and
\(\Delta_3\otimes A^+\) is totally unimodular.
\end{lemma}
\begin{proof}
Choose a reduced incidence matrix \(R\) with
\(\row(R)=\row(A)\), and orient \(f\) in the same direction as
\(e\).  Writing \(A=XR\) with \(X\) nonsingular gives
\(A^+=X[R\mid\boldsymbol c_e]\), where
\(\boldsymbol c_e\) is the \(e\)-column of \(R\).
The matrix \([R\mid\boldsymbol c_e]\) is a reduced incidence
matrix of \(G^+\), so \(A^+\) has the required row space and rank.
The new columns of \(\Delta_3\otimes A^+\) duplicate the columns
of \(\Delta_3\otimes A\) labelled by \((j,e)\), \(j\in[3]\).
Each square minor is therefore zero or, up to sign, a square minor
of \(\Delta_3\otimes A\).  This proves total unimodularity.
\end{proof}

For edge subdivisions, we use the Ghouila--Houri characterization
\cite{ghouila-houri}; see also
\cite[Corollary~4.7]{conforti-integer-programming}.  A matrix
\(K=(k_{ij})\in\{0,1,-1\}^{p\times q}\) is totally unimodular if
and only if, for every \(S\subseteq[p]\), there exist signs
\(\varepsilon_i\in\{\pm1\}\), \(i\in S\), such that
\begin{equation}\label{eq:ghouila-houri}
  \sum_{i\in S}\varepsilon_i k_{ij}\in\{0,1,-1\}
  \qquad\text{for every }j\in[q].
\end{equation}
\begin{lemma}
\label{lem:triangle-subdivision-tu}
Let \(G\) be a loopless graph and let \(G^+\) be obtained by
subdividing \(e\in E(G)\) into \(e_1,e_2\).  Suppose that
\[
 A=[A_0\mid\boldsymbol a_e],\qquad
 A_0=A_{E(G)\setminus\{e\}},
\]
is an integral full row rank incidence representation of \(M(G)\)
over \(\QQ\).  If \(\Delta_3\otimes A\) is totally unimodular, then
\begin{equation}\label{eq:subdivision-incidence-representation}
 A^+:=
 \begin{bmatrix}
  A_0&\boldsymbol a_e&0\\
  0&1&-1
 \end{bmatrix},
\end{equation}
with the original labels on \(A_0\) and the last two columns
labelled by \(e_1,e_2\), is an integral full row rank incidence
representation of \(M(G^+)\) over \(\QQ\), and
\(\Delta_3\otimes A^+\) is totally unimodular.
\end{lemma}

\begin{proof}
Write \(e=uv\), with \(e_1=uw\) and \(e_2=wv\).  Choose an
orientation \(o(G)\) with \(\row(A)=\row(A_{o(G)})\) and \(e\)
directed from \(u\) to \(v\), reversing all edge directions if
necessary.  This reversal negates the incidence matrix and preserves
its row space.  Choose a reduced incidence matrix
\(R=[R_0\mid\boldsymbol c_e]\) by deleting the row at \(v\)
and one vertex row from each other component.  Then \(A=XR\)
for a nonsingular matrix \(X\in\QQ^{r\times r}\), where
\(r=r(G)\).  Orient the subdivided path as \(u\to w\to v\),
and delete the same old vertex rows.  The resulting reduced
incidence matrix satisfies
\[
 \widetilde R^+=
 \begin{bmatrix}
  R_0&\boldsymbol c_e&0\\
  0&-1&1
 \end{bmatrix},
 \qquad
 A^+=\operatorname{diag}(X,-1)\widetilde R^+.
\]
Thus \(A^+\) has the required row space and full row rank.

For \(g\in E(G)\), let \(\boldsymbol a_g\) be the \(g\)-column of
\(A\).  Write \(\Delta_3=[\boldsymbol d_1\mid\boldsymbol d_2\mid\boldsymbol d_3]\),
and set \(K=\Delta_3\otimes A\) and
\(K^+=\Delta_3\otimes A^+\).  Move the two new rows of \(K^+\)
to the bottom.  Its columns then have the following forms, for
\(j\in[3]\):
\begin{equation}\label{eq:subdivision-tensor-block-matrix}
 \begin{array}{c@{\qquad}c@{\qquad}c}
  (j,g),\ g\ne e &(j,e_1)&(j,e_2)\\[2mm]
  \displaystyle\binom{\boldsymbol d_j\otimes\boldsymbol a_g}{\boldsymbol 0_2}
  &\displaystyle\binom{\boldsymbol d_j\otimes\boldsymbol a_e}{\boldsymbol d_j}
  &\displaystyle\binom{\boldsymbol 0_{2r}}{-\boldsymbol d_j}
 \end{array}
\end{equation}
Since \(A\) is a submatrix of the TU matrix \(K\), this display
also shows that every entry of \(K^+\) lies in \(\{0,1,-1\}\).

Let \(S\) be a set of rows of \(K^+\), with
\(S_0=S\cap[2r]\) and \(S_1=S\setminus S_0\).
Choose a Ghouila--Houri signing of \(S_0\) in \(K\), and let
\(s_j\) be the signed sum in the column
\(\boldsymbol d_j\otimes\boldsymbol a_e\).
Since \(s_j\in\{0,1,-1\}\) and
\(\boldsymbol d_1+\boldsymbol d_2+\boldsymbol d_3=0\),
\(\boldsymbol s=(s_1,s_2,s_3)\) is zero or a permutation of
\((1,-1,0)\).

Sign the rows in \(S_1\), using opposite signs when both are
present.  Their signed-sum vector \(\boldsymbol u\) in
\(\Delta_3\) is likewise zero or a permutation of \((1,-1,0)\).
Reversing all these signs if necessary, we may arrange that
\(\boldsymbol s+\boldsymbol u\in\{0,1,-1\}^3\).
Indeed, if both vectors are nonzero, their supports intersect.
Choose opposite signs at one common coordinate: equal supports
then give \(\boldsymbol u=-\boldsymbol s\), while unequal
supports share only that coordinate.  If either vector is zero,
no adjustment is needed.

By \eqref{eq:subdivision-tensor-block-matrix}, the signed sums
in columns \((j,e_1)\) and \((j,e_2)\) are \(s_j+u_j\) and
\(-u_j\), respectively; all other column sums agree with those
in \(K\).  Every sum therefore lies in \(\{0,1,-1\}\), so
\eqref{eq:ghouila-houri} proves that \(K^+\) is totally unimodular.
\end{proof}

The definition of a series--parallel network and the preceding
lemmas give the following regularity result.
\begin{proposition}
\label{prop:triangle-series-parallel}
Let \(G\) be a \(2\)-connected simple series--parallel network, and \(\F\) be an arbitrary field.  Then \(M_{\F}(C_3,G)\) is regular.
\end{proposition}
\begin{proof}
Choose a construction sequence \(G_0=C_2,G_1,\ldots,G_t=G\),
where each step is an edge subdivision or a parallel-edge addition.
Orient the two edges of \(C_2\) in the same direction and take
\(A_0=[1\ 1]\).  The matrix \(\Delta_3\otimes A_0\) is totally
unimodular, since it is obtained by duplicating each column of
\(\Delta_3\).

Applying Lemmas~\ref{lem:triangle-parallel-tu}
and~\ref{lem:triangle-subdivision-tu} at each step produces an
integral full row rank incidence representation \(A_t\) of
\(M(G)\) over \(\QQ\) such that
\(K_t=\Delta_3\otimes A_t\) is totally unimodular.
The matrix \(A_t\) is a submatrix of \(K_t\), so it is also
totally unimodular.  Every square minor of either matrix is
\(0\), \(1\), or \(-1\); consequently, its rank on every
column set is unchanged when its entries are interpreted over any
field.  Thus \(A_t\) represents \(M(G)\) over \(\F\), and
\(K_t\) represents the same regular matroid over every field.
Since \(K_t=\Delta_3\otimes A_t\),
Proposition~\ref{thm:basic-represented-product} identifies this
matroid with \(M_{\F}(C_3,G)\).
\end{proof}

\section{Classification theorems}
\label{sec:regular-matroids}
We first classify products of graphic factors using the reduction
and constructions of Section~\ref{sec:positive-constructions}.
This gives the \mainthm{} for binary factors over a common field.  We then derive
the factor-minor characterization and treat higher products,
loops, and parallel elements.

\subsection{Classification for graphic factors}\label{subsec:graphic-classification}
Let \(G,H\) be simple graphs with nonempty edge sets, and let
\(\F\) be a field.  If one factor is a forest, say \(G\), then
\eqref{eq:free-factor} gives
\[
M_{\F}(G,H)\cong M(H)^{\oplus |E(G)|}.
\]
Thus the product is graphic and regular.  Since \(|E(G)|\ge1\),
closure under direct sums and minors shows that it is cographic
if and only if \(M(H)\) is cographic, or equivalently, \(H\) is
planar by Fact~\ref{fact:whitney-planarity}.

It remains to consider graphs that both contain cycles.
By Theorem~\ref{thm:graphic-reduction}, regularity excludes an
\(M(K_4)\)-minor from either factor matroid.  We first combine
this restriction with the constructions for a cycle and a single
cyclic block.
\begin{theorem}
\label{thm:local-classification}
Let \(m\ge3\), \(\F\) be a field, and \(G\) be a \(2\)-connected simple graph.  Then \(M_{\F}(C_m,G)\) is regular if and only if
\[
 \begin{cases}
 G\text{ is a series--parallel network},&m=3,\\
 G\text{ is outerplanar},&m\ge4.
 \end{cases}
\]
Moreover, \(M_{\F}(C_m,G)\) is cographic if and only if \(G\) is outerplanar.
\end{theorem}
\begin{proof}
If \(G\) is outerplanar, Proposition~\ref{prop:outerplanar-cographic} identifies
\(M_{\F}(C_m,G)\cong M^*(X_{m,G})\); the product is therefore cographic and,
by Fact~\ref{fact:matroid-classes}, regular.  If \(m=3\)
and \(G\) is series--parallel, regularity follows from
Proposition~\ref{prop:triangle-series-parallel}.  This proves sufficiency.

Conversely, suppose \(M_{\F}(C_m,G)\) is regular.
Theorem~\ref{thm:graphic-reduction} excludes an \(M(K_4)\)-minor
from \(M(G)\).  Thus \(G\) is series--parallel by
Fact~\ref{fact:series-parallel-minors} and has no \(K_4\)-minor by
Fact~\ref{fact:graph-minors}.  This proves necessity when \(m=3\).
Let \(m\ge4\), and suppose that \(G\) is not outerplanar.
Fact~\ref{fact:outerplanar-minors} then gives a \(K_{2,3}\)-minor.
Contracting \(m-4\) edges of \(C_m\) and applying
Proposition~\ref{thm:factor-minors} gives an
\(M_{\F}(C_4,K_{2,3})\)-minor of \(M_{\F}(C_m,G)\).
This contradicts Proposition~\ref{prop:specified-minors}\textup{(i)}
and the minor-closedness of the regular class.  Hence \(G\) is outerplanar.

For cographicity, suppose \(M_{\F}(C_m,G)\) is cographic.
It is regular by Fact~\ref{fact:matroid-classes}, so the case
\(m\ge4\) follows from the necessity just proved.  If \(m=3\)
and \(G\) is not outerplanar, the exclusion of \(K_4\) above
and Fact~\ref{fact:outerplanar-minors} give a \(K_{2,3}\)-minor.
Fact~\ref{fact:graph-minors} and
Proposition~\ref{thm:factor-minors} then give an
\(M_{\F}(C_3,K_{2,3})\)-minor of \(M_{\F}(C_3,G)\),
contradicting Proposition~\ref{prop:specified-minors}\textup{(ii)}
and the minor-closedness of the cographic class.  Thus \(G\) is outerplanar.
\end{proof}

For every field \(\F\), the product \(M_{\F}(C_3,K_{2,3})\) is
regular but not cographic by
Propositions~\ref{prop:triangle-series-parallel} and
\ref{prop:specified-minors}\textup{(ii)}.
Replacing \(C_3\) by \(C_4\) gives a nonregular product by
Proposition~\ref{prop:specified-minors}\textup{(i)}.

The next lemma is used to show that at least one factor graph is a cactus.
Its graph-theoretic assertion appears in \cite[proof of Theorem~5.2]{fife-oxley}, and its matroidal assertion follows
from Fact~\ref{fact:graph-minors}.
\begin{lemma}
\label{lem:diamond-minor-in-block}
Let \(G\) be a \(2\)-connected simple graph.  If \(G\) is not a cycle, then \(G\) has \(K_4-e\) as a graph minor, and \(M(G)\) has an
\(M(K_4-e)\)-minor.
\end{lemma}
We shall use the standard block characterization of outerplanar graphs \cite[Theorem~11.8]{harary-graph-theory}:
\[
H \text{ is outerplanar}\quad\Longleftrightarrow\quad \text{every block of }H\text{ is outerplanar}.
\]
The local theorem and the block decomposition combine in the global classification.
\begin{theorem}\label{thm:global-classification}
Let \(G,H\) be simple graphs, each containing a cycle, and let
\(\F\) be a field.
\begin{enumerate}[label=\textup{(\roman*)},leftmargin=2.5em]
\item The matroid \(M_{\F}(G,H)\) is regular if and only if,
after possibly interchanging \(G\) and \(H\), one of the following
holds:
\begin{enumerate}[label=\textup{(\alph*)},leftmargin=2.5em]
\item \(G\) is a cactus and \(H\) is outerplanar;
\item every cyclic block of \(G\) is a triangle, and every cyclic
block of \(H\) is a series--parallel network.
\end{enumerate}
\item The matroid \(M_{\F}(G,H)\) is cographic if and only if,
after possibly interchanging \(G\) and \(H\), \(G\) is a cactus
and \(H\) is outerplanar.
\end{enumerate}
The product is not graphic.  It is binary if and only if
\(\operatorname{char}(\F)=2\) or it is regular.
\end{theorem}
\begin{proof}
Suppose first that \(M_{\F}(G,H)\) is regular.
Theorem~\ref{thm:graphic-reduction} excludes an \(M(K_4)\)-minor
from each factor matroid.  Hence every cyclic block of either graph
is a series--parallel network by Fact~\ref{fact:series-parallel-minors}.
At least one factor is a cactus.  Otherwise, each graph has a
cyclic block that is not a cycle.  Applying
Lemma~\ref{lem:diamond-minor-in-block} to these blocks gives an
\(M(K_4-e)\)-minor in each factor matroid, and
Proposition~\ref{thm:factor-minors} gives an
\(M_{\F}(K_4-e,K_4-e)\)-minor of \(M_{\F}(G,H)\).
This contradicts Proposition~\ref{prop:specified-minors}\textup{(i)}
and the minor-closedness of the regular class.

After interchanging the factors, assume that \(G\) is a cactus.
If \(G\) has a cyclic block \(C_m\) with \(m\ge4\), then for
every cyclic block \(H_j\) of \(H\),
Proposition~\ref{prop:block-decomposition} identifies
\(M_{\F}(C_m,H_j)\) as a direct summand of \(M_{\F}(G,H)\).
It is regular, so Theorem~\ref{thm:local-classification} implies
that \(H_j\) is outerplanar.  The remaining nonempty blocks are
bridges; hence \(H\) is outerplanar, giving \textup{(i)(a)}.  If all cyclic blocks of
\(G\) are triangles, the series--parallel condition already proved
for \(H\) gives \textup{(i)(b)}.

Conversely, suppose one of the conditions in~\textup{(i)} holds.
In either case \(G\) is a cactus.  By
Proposition~\ref{prop:block-decomposition}, it suffices to check
the summand for each pair of blocks.  For a pair of cyclic blocks
\(C_m\) of \(G\) and \(H_j\) of \(H\), condition~\textup{(a)}
makes \(H_j\) outerplanar, whereas condition~\textup{(b)} gives
\(m=3\) and \(H_j\) series--parallel.
Theorem~\ref{thm:local-classification} proves regularity in both
cases.  A pair containing a bridge contributes the cycle matroid
of the other block or \(U_{1,1}\), and is also regular.
Closure under direct sums proves~\textup{(i)}.

For~\textup{(ii)}, suppose the product is cographic.  It is
regular, so~\textup{(i)} allows us to assume that \(G\) is a cactus.
Fix a cyclic block \(C_m\) of \(G\).  For each cyclic block
\(H_j\) of \(H\), the summand \(M_{\F}(C_m,H_j)\) is cographic,
and Theorem~\ref{thm:local-classification} makes \(H_j\) outerplanar.
Thus \(H\) is outerplanar.
Conversely, if \(G\) is a cactus and \(H\) is outerplanar, then
each pair of cyclic blocks gives a cographic summand by
Theorem~\ref{thm:local-classification}.  A pair containing a bridge
contributes the cycle matroid of a planar block or \(U_{1,1}\).
Every summand is therefore cographic, which proves~\textup{(ii)}.

Finally, each graph has a \(C_3\)-minor.  By
Proposition~\ref{thm:factor-minors} and \eqref{eq:two-cycle-identity},
the product has an \(M^*(K_{3,3})\)-minor, which is not graphic
by Fact~\ref{fact:whitney-planarity}.  The product is therefore
not graphic.  The binary assertion follows from
Lemma~\ref{lem:prime-field-reduction} and
Fact~\ref{fact:matroid-classes}.
\end{proof}

\begin{example}
Let \(G\) be the graph obtained by identifying one vertex of \(C_3\)
with one vertex of \(C_4\).  The cyclic blocks of \(G\) are \(C_3\) and \(C_4\).  Proposition~\ref{prop:block-decomposition} and \eqref{eq:two-cycle-identity} identify the decomposition
\[
 M_{\F}(G,C_5)= M_{\F}(C_3,C_5)\oplus M_{\F}(C_4,C_5)\cong M^*(K_{3,5})\oplus M^*(K_{4,5}).
\]
Thus \(M_{\F}(G,C_5)\) is cographic and, by
Theorem~\ref{thm:global-classification}, not graphic.
\end{example}
\subsection{The Main Theorem and its consequences}\label{subsec:regular-two-factor-proof}
\begin{proof}[Proof of the \mainthm{}]
If one factor is free, interchange the factors if necessary and write
\(M=U_{q,q}\).  Then
\(T_{\F}(M,N)\cong N^{\oplus q}\) by
Proposition~\ref{prop:tensor-direct-sums}.  The classes of regular,
graphic, and cographic matroids are closed under direct sums and
minors.  Since \(q\ge1\), the product belongs to each of these
classes exactly when \(N\) does.

Suppose now that both factors are nonfree.  If the product is
regular, Theorem~\ref{thm:graphic-reduction} implies that both
factors are graphic.  Choose simple graph representations
\(M=M(G)\), \(N=M(H)\); the tensor products agree by
Proposition~\ref{thm:basic-represented-product}.
Theorem~\ref{thm:global-classification}\textup{(i)} and
Lemma~\ref{lem:regular-sp-outerplanar} yield \textup{(R1)} or
\textup{(R2)}, after possibly interchanging the factors.
Conversely, either condition supplies simple graph representations
of the required kinds by the lemma, and the same theorem gives
regularity.  This proves~(i).

If the product is cographic, it is regular.  In the present nonfree
case, we may therefore use simple graphic representatives \(G,H\).
Theorem~\ref{thm:global-classification}\textup{(ii)} and
Lemma~\ref{lem:regular-sp-outerplanar} yield \textup{(R1)}, after
possibly interchanging the factors.  Conversely, under
\textup{(R1)}, Lemma~\ref{lem:regular-sp-outerplanar} supplies
simple graph representations with \(G\) a cactus and \(H\)
outerplanar, so Theorem~\ref{thm:global-classification}\textup{(ii)}
gives cographicity.  This proves~(ii).

For~(iii), two nonfree factors contain \(U_{2,3}\)-minors, so their
product has \(T_{\F}(U_{2,3},U_{2,3})\cong M^*(K_{3,3})\) as a minor
by Proposition~\ref{thm:factor-minors} and
\eqref{eq:two-cycle-identity}.  This minor is not graphic by
Fact~\ref{fact:whitney-planarity}, proving~(iii).

Part~(iv) follows from Lemma~\ref{lem:prime-field-reduction}
in characteristic two.  In other characteristics, the product is
\(\F\)-representable, so the binary--regular equivalence in
Fact~\ref{fact:matroid-classes} applies.
\end{proof}

For connected factors, the cactus condition reduces to a single
circuit component.
\begin{corollary}\label{cor:regular-connected}
Let \(M,N\) be connected simple nonfree binary matroids representable
over a common field \(\F\).
Then \(T_{\F}(M,N)\) is regular if and only if, after possibly interchanging
the factors, \(M\cong U_{m-1,m}\) for some \(m\ge3\), and
\[
 \begin{cases}
 N\text{ is series--parallel},&m=3,\\
 N\text{ is outerplanar},&m\ge4.
 \end{cases}
\]
It is cographic if and only if, after possibly interchanging the factors,
\(M\cong U_{m-1,m}\) for some \(m\ge3\) and \(N\) is outerplanar.
It is not graphic.
\end{corollary}
\begin{proof}
Under \textup{(R1)}, the connected nonfree cactus factor is a
circuit matroid; under \textup{(R2)}, it is \(U_{2,3}\).
Since every outerplanar matroid is series--parallel, part~(i) of
the \mainthm{} gives exactly the stated regular cases, up to
interchange of the factors.  Parts~(ii) and~(iii) give the
remaining assertions.
\end{proof}

The four two-factor products in
Proposition~\ref{prop:specified-minors}\textup{(i)--(ii)} also give
an intrinsic characterization by factor minors.
\begin{corollary}\label{cor:regular-forbidden-pairs}
Let \(M,N\) be simple nonfree binary matroids representable over
a common field \(\F\).  Then \(T_{\F}(M,N)\) is regular if and only if none of the pairs
\[
 (U_{2,3},M(K_4)),\qquad
 (M(K_4-e),M(K_4-e)),\qquad
 (U_{3,4},M(K_{2,3}))
\]
occurs as factor minors of \((M,N)\) in either order.
Here \(e\in E(K_4)\) and \(U_{3,4}\cong M(C_4)\).
The product is cographic if and only if the first two pairs and
\((U_{2,3},M(K_{2,3}))\) are all excluded in either order.
\end{corollary}
\begin{proof}
If a listed pair occurs as factor minors, Proposition~\ref{thm:factor-minors}
gives its represented product as a minor of \(T_{\F}(M,N)\).
Proposition~\ref{prop:specified-minors} proves the necessity of the
respective exclusions.

For regularity, suppose none of the three displayed pairs occurs in either order.  Each simple nonfree factor has a
\(U_{2,3}\)-minor.  Exclusion of the first pair therefore forces
both factors to avoid \(M(K_4)\).  By
Lemma~\ref{lem:regular-sp-outerplanar}, choose simple graphic
representations \(M=M(G)\) and \(N=M(H)\), whose cyclic blocks
are series--parallel.  At least one graph is a cactus: otherwise
Lemma~\ref{lem:diamond-minor-in-block} gives an \(M(K_4-e)\)-minor
in each factor.  Interchange the factors so that \(G\) is a cactus.
If every cyclic block of \(G\) is a triangle, condition
\textup{(R2)} of the \mainthm{} holds.
Otherwise \(M(G)\) has a \(U_{3,4}\)-minor.  The third excluded
pair forces \(M(H)\) to avoid \(M(K_{2,3})\), so condition
\textup{(R1)} holds.  The main theorem gives regularity in either case.

For cographicity, suppose all three prescribed pairs are excluded.
The exclusion of \((U_{2,3},M(K_{2,3}))\) also rules out
\((U_{3,4},M(K_{2,3}))\), since \(U_{2,3}\) is a minor of
\(U_{3,4}\).  The preceding argument therefore gives simple graph
representations \(M=M(G)\) and \(N=M(H)\) with \(G\) a cactus.
The \(U_{2,3}\)-minor in \(M(G)\) now forces \(M(H)\) to avoid
\(M(K_{2,3})\).
Thus \textup{(R1)} holds and the product is cographic by
part~\textup{(ii)} of the \mainthm{}.
\end{proof}

\subsection{Higher products and arbitrary factors}\label{subsec:higher-products}
We classify higher products by the number of nonfree factors.
\begin{theorem}\label{thm:regular-higher}
Let \(\F\) be a field, \(s\ge2\), and \(M_1,\ldots,M_s\) be simple
binary matroids with nonempty ground sets, all representable over \(\F\).  Set \(I=\{i:M_i\text{ is nonfree}\}\) and
\(q=\prod_{i\notin I}|E(M_i)|\), where an empty product is one.
Let \(P_I\) be the represented tensor product over \(\F\) of the factors indexed by \(I\),
with \(P_\varnothing=U_{1,1}\) and \(P_{\{j\}}=M_j\).  Then
\[
 T_{\F}(M_1,\ldots,M_s)\cong P_I^{\oplus q}.
\]
\begin{enumerate}[label=\textup{(\roman*)},leftmargin=2.5em,itemsep=0pt,parsep=0pt,topsep=3pt]
\item If \(I=\varnothing\), the product is free.
\item If \(I=\{j\}\), then for each of the classes of regular,
graphic, and cographic matroids, the product belongs to that class
if and only if \(M_j\) does.
\item If \(|I|=2\), write \(I=\{i,j\}\).  The product is not graphic.  It is
regular if and only if \(M_i,M_j\) satisfy \textup{(R1)} or
\textup{(R2)} of the \mainthm{}, and cographic if and only if
they satisfy \textup{(R1)}, in each case allowing interchange
of \(M_i\) and \(M_j\).
\item If \(|I|\ge3\), the product is nonregular, and hence
neither graphic nor cographic.
\end{enumerate}
In characteristic two the product is binary; in other
characteristics it is binary if and only if it is regular.
\end{theorem}
\begin{proof}
Apply the free-factor identity in
Proposition~\ref{prop:tensor-direct-sums} to each free coordinate
to obtain the direct-sum decomposition.  Since \(q\ge1\) and
the four classes are closed under minors and direct sums,
class membership is determined by \(P_I\).
The cases \(|I|\le2\) follow from
the definitions of \(P_\varnothing\) and \(P_{\{j\}}\), and the \mainthm{}.
If \(|I|\ge3\), choose three nonfree factors and take their
\(U_{2,3}\)-minors; restrict every remaining factor to a nonloop.
Proposition~\ref{thm:factor-minors} gives
\(T_{\F}(U_{2,3},U_{2,3},U_{2,3})\) as a minor.  It is nonregular
over every field by Proposition~\ref{prop:specified-minors}\textup{(iii)}.
The binary assertion follows from the prime-field reduction and the
binary--regular criterion, as in the \mainthm{}.
\end{proof}

\begin{corollary}\label{cor:repeated-factors}
Let \(M\) be a simple binary matroid with a nonempty ground set,
representable over a field \(\F\).
\begin{enumerate}[label=\textup{(\roman*)},leftmargin=2.5em]
\item \(T_{\F}(M,M)\) is graphic if and only if \(M\) is free.
Each of regularity and cographicity is equivalent to \(M\) being
a cactus matroid.
\item For \(s\ge3\), each of graphicity, regularity, and
cographicity of \(T_{\F}(M,\ldots,M)\) is equivalent to \(M\)
being free.
\end{enumerate}
All these products are binary in characteristic two; in other
characteristics they are binary exactly in the listed regular cases.
\end{corollary}
\begin{proof}
If \(M\) is free, all the products considered are free, and
\(M\) is a cactus matroid.  Assume henceforth that \(M\) is
nonfree.  For \textup{(i)}, the \mainthm{} excludes graphicity
and forces \(M\) to be a cactus matroid whenever the product
is regular.  Conversely, a cactus matroid is outerplanar, since
it is a direct sum of circuit matroids and coloops, so
\textup{(R1)} gives cographicity and hence regularity.
For \textup{(ii)}, all \(s\) factors are nonfree, so
Theorem~\ref{thm:regular-higher} excludes regularity and hence
graphicity and cographicity.  The binary assertions follow from
Lemma~\ref{lem:prime-field-reduction} and
Fact~\ref{fact:matroid-classes}.
\end{proof}

Finally, let \(s\ge2\), and let \(M_1,\ldots,M_s\) be arbitrary
finite binary matroids representable over a common field \(\F\).
Proposition~\ref{prop:regular-simplification} reduces each of the
four class-membership questions to the product of their
simplifications.  If one simplification is empty, the product is
an all-loop matroid, possibly empty, and belongs to all four
classes.  Otherwise the \mainthm{} and
Theorem~\ref{thm:regular-higher} apply to the simplifications.

\section{Proof of Proposition~\ref*{prop:specified-minors}}\label{sec:minor-proofs}
This section proves Proposition~\ref{prop:specified-minors} by
explicit deletion and contraction calculations.
\subsection{Contraction criteria}
All integer matrices are interpreted over \(\F\) as in
Section~\ref{sec:preliminaries}.  Rows of each Kronecker product
are in lexicographic order, with the first factor coordinate first;
row and column indices start at \(1\).  Whenever a set of column
indices is listed, the displayed order is used in the corresponding
submatrix.
For four retained columns, determinant lists use the pair order
\((1,2),(1,3),(1,4),(2,3),(2,4),(3,4)\).
In the characteristic-two cases below, all computations are over
\(\operatorname{GF}(2)\), by Lemma~\ref{lem:prime-field-reduction}.
The following lemmas turn the specified matrices into minor containments.
\begin{lemma}
\label{prop:u24-contraction}
Let \(K\) be a full row rank matrix over a field \(\F\) with
\(r\ge2\) rows and columns labelled by \(E\).  Let
\(C\subseteq E\) have \(r-2\) elements and
\(R\subseteq E\setminus C\) have four elements.  Write
\(K_R=[\boldsymbol{k}_1|\boldsymbol{k}_2|\boldsymbol{k}_3|\boldsymbol{k}_4]\).
Suppose that \(\rank(K_C)=r-2\) and there is a rank-two matrix
\(L\in\F^{2\times r}\) such that
\[
LK_C=0,\qquad \det[L\boldsymbol{k}_i\,|\,L\boldsymbol{k}_j]\ne0
\quad(1\le i<j\le4).
\]
Then \((M[K]/C)|R\cong U_{2,4}\).
\end{lemma}
\begin{proof}
The rank assumptions and \(LK_C=0\) give
\(\ker(L)=\col(K_C)\).  Hence, for \(X\subseteq E\setminus C\)
and \(\boldsymbol\alpha\in\F^X\),
\[
 LK_X\boldsymbol\alpha=0
 \quad\Longleftrightarrow\quad
 K_X\boldsymbol\alpha=K_C\boldsymbol\beta
 \text{ for some }\boldsymbol\beta\in\F^C.
\]
Since the columns of \(K_C\) are independent, the columns of
\(LK_X\) are independent if and only if those of \(K_{C\cup X}\)
are independent.  This is exactly the condition that \(X\) be
independent in \(M[K]/C\)
\cite[Proposition~3.1.7]{oxley}.  Thus
\((M[K]/C)|R=M[LK_R]\).
The six nonzero \(2\times2\) determinants show that every two
columns of \(LK_R\) are independent, so this rank-two matroid
on four elements is \(U_{2,4}\).
\end{proof}
In characteristic two, the represented products considered here
are binary, so they have no \(U_{2,4}\)-minor.  The next lemma
detects an \(F_7\)-minor.
\begin{lemma}
\label{prop:f7-contraction}
Let \(K\) be a full row rank matrix over \(\operatorname{GF}(2)\)
with \(r\ge3\) rows and columns labelled by \(E\).  Let
\(C\subseteq E\) have \(r-3\) elements and
\(R\subseteq E\setminus C\) have seven elements.  Suppose that
\(\rank(K_C)=r-3\) and there is a rank-three matrix
\(L\in\operatorname{GF}(2)^{3\times r}\) such that \(LK_C=0\)
and the seven columns of \(LK_R\) are precisely the nonzero vectors
of \(\operatorname{GF}(2)^3\).  Then
\[
 (M[K]/C)|R\cong F_7.
\]
\end{lemma}
\begin{proof}
The columns of \(K_C\) are independent and
\(\col(K_C)\subseteq\ker(L)\).  Both spaces have dimension
\(r-3\), so they are equal.  The linear-dependence argument in
the proof of Lemma~\ref{prop:u24-contraction} gives
\((M[K]/C)|R=M[LK_R]\).  The seven columns of \(LK_R\) are
precisely the nonzero vectors of \(\operatorname{GF}(2)^3\), which form the
standard representation of \(F_7\).
\end{proof}

\subsection{The pair \texorpdfstring{\((C_3,K_4)\)}{(C3,K4)}}
\begin{proof}[Proof of Proposition~\ref{prop:specified-minors}\textup{(i)} for this pair]
Label the edges of \(C_3\) by \(1,2,3\) and those of \(K_4\) by
\(12,13,14,23,24,34\), in that order.  Use the reduced incidence matrices
\[
 A_{C_3}=\begin{bmatrix}1&-1&0\\0&1&-1\end{bmatrix},
 \qquad
 A_{K_4}=\begin{bmatrix}
 1&1&1&0&0&0\\-1&0&0&1&1&0\\0&-1&0&-1&0&1
 \end{bmatrix}.
\]
Put \(K=A_{C_3}\otimes A_{K_4}\).  Then \(\rank(K)=6\) and
\(M[K]=M_{\F}(C_3,K_4)\).

Suppose first that \(\operatorname{char}(\F)\ne2\), and choose
\[
\begin{aligned}
 C&=\{(2,12),(2,34),(3,13),(3,24)\},\\
 R&=\{(1,12),(1,13),(1,14),(1,23)\}.
\end{aligned}
\]
Take
\[
 L=\begin{bmatrix}
1&1&0&0&0&0\\
0&-1&1&1&0&1
\end{bmatrix}.
\]
Direct verification gives
\[
 \rank(K_C)=4,\qquad \rank(L)=2,\qquad LK_C=0.
\]
The minor of \(K_C\) on rows \(1,3,4,5\) is \(1\), and the
six column-pair determinants of \(LK_R\) are
\(-1,-1,-1,1,-1,-2\).  They are all nonzero in \(\F\), so
Lemma~\ref{prop:u24-contraction} yields
\((M[K]/C)|R\cong U_{2,4}\).

In characteristic two, choose
\[
\begin{aligned}
 C&=\{(1,12),(1,34),(2,13)\},\\
 R&=\{(2,14),(3,24),(3,23),(3,14),(1,13),(2,12),(2,24)\}.
\end{aligned}
\]
Take
\[
 L=\begin{bmatrix}
1&1&0&1&0&0\\
0&0&0&0&1&0\\
1&1&0&0&0&1
\end{bmatrix}.
\]
Here
\[
 \rank(K_C)=3,\qquad \rank(L)=3,\qquad LK_C=0;
\]
the minor of \(K_C\) on rows \(1,2,3\) is \(1\), and the columns
of \(LK_R\) are precisely the seven nonzero vectors of
\(\operatorname{GF}(2)^3\).  Lemma~\ref{prop:f7-contraction} gives
\((M[K]/C)|R\cong F_7\).  In either characteristic,
Fact~\ref{fact:regular-excluded} proves nonregularity.
\end{proof}

\subsection{The pair \texorpdfstring{\((K_4-e,K_4-e)\)}{(K4-e,K4-e)}}
\begin{proof}[Proof of Proposition~\ref{prop:specified-minors}\textup{(i)} for this pair]
Take \(V(K_4-e)=\{1,2,3,4\}\), with \(e=14\), and orient
all remaining edges from the smaller endpoint to the larger.
Label \(12,13,23,24,34\) by \(1,2,3,4,5\), respectively.
Deleting the row of vertex \(4\) gives
\[
 A_{K_4-e}=\begin{bmatrix}
 1&1&0&0&0\\-1&0&1&1&0\\0&-1&-1&0&1
 \end{bmatrix}.
\]
For \(K=A_{K_4-e}\otimes A_{K_4-e}\), we have
\(\rank(K)=9\) and \(M[K]=M_{\F}(K_4-e,K_4-e)\).

If \(\operatorname{char}(\F)\ne2\), choose
\[
\begin{aligned}
 C&=\{(4,5),(5,1),(3,3),(5,4),(1,2),(2,5),(4,2)\},\\
 R&=\{(1,1),(1,4),(1,5),(2,1)\}.
\end{aligned}
\]
Take
\[
 L=\begin{bmatrix}
0&1&0&0&0&0&0&0&0\\
1&0&1&0&-1&0&0&0&1
\end{bmatrix}.
\]
Then
\[
 \rank(K_C)=7,\qquad \rank(L)=2,\qquad LK_C=0.
\]
The minor of \(K_C\) on rows \(1,3,4,5,6,7,8\) is \(1\).
The six column-pair determinants of \(LK_R\) are
\(-1,-1,-1,1,2,1\), so Lemma~\ref{prop:u24-contraction} gives
\((M[K]/C)|R\cong U_{2,4}\).

In characteristic two, choose
\[
\begin{aligned}
 C&=\{(1,4),(1,1),(2,5),(5,2),(4,5),(4,1)\},\\
 R&=\{(1,2),(5,4),(2,1),(3,2),(2,2),(2,3),(3,3)\}.
\end{aligned}
\]
Take
\[
 L=\begin{bmatrix}
1&1&0&1&1&0&0&0&0\\
0&0&0&0&0&0&0&1&0\\
0&0&1&0&0&0&1&0&1
\end{bmatrix}.
\]
The identities are
\[
 \rank(K_C)=6,\qquad \rank(L)=3,\qquad LK_C=0.
\]
The minor of \(K_C\) on rows \(1,2,3,4,6,7\) is \(1\), and the
columns of \(LK_R\) are precisely the nonzero vectors of
\(\operatorname{GF}(2)^3\).  Thus
\((M[K]/C)|R\cong F_7\) by Lemma~\ref{prop:f7-contraction}.
Fact~\ref{fact:regular-excluded} proves nonregularity over every field.
\end{proof}

\subsection{The pairs involving \texorpdfstring{\(K_{2,3}\)}{K2,3}}
Write \(V(K_{2,3})=\{x_1,x_2\}\sqcup\{y_1,y_2,y_3\}\),
label \(x_i y_j\) by \(ij\), and orient every edge from \(x_i\)
to \(y_j\).  Delete the \(y_3\)-row and use the column order
\(11,12,13,21,22,23\).  The reduced incidence matrix is
\[
 A_{K_{2,3}}=\begin{bmatrix}
 1&1&1&0&0&0\\0&0&0&1&1&1\\-1&0&0&-1&0&0\\0&-1&0&0&-1&0
 \end{bmatrix}.
\]

\begin{proof}[Proof of Proposition~\ref{prop:specified-minors}\textup{(ii)}]
Put \(K=\Delta_3\otimes A_{K_{2,3}}\), so that
\(\rank(K)=8\) and \(M[K]=M_{\F}(C_3,K_{2,3})\), and choose
\[
\begin{aligned}
 C&=\{(1,11),(1,12),(1,13)\},\\
 R&=\{(2,11),(2,12),(2,13),(2,21),(3,21),(2,22),(3,22),(2,23),(3,23)\}.
\end{aligned}
\]
Take
\[
 L=\begin{bmatrix}
0&0&0&0&0&-1&-1&0\\
0&0&0&0&0&0&0&-1\\
0&0&0&0&1&1&1&1\\
0&-1&0&0&0&1&0&0\\
0&-1&0&0&0&0&0&0
\end{bmatrix}.
\]
We have
\[
 \rank(K_C)=3,\qquad \rank(L)=5,\qquad LK_C=0.
\]
Indeed, the minor of \(K_C\) on rows \(1,3,4\) is \(1\), and
the minor of \(L\) in columns \(2,5,6,7,8\) is \(-1\).

To identify the retained columns, give \(K_{3,3}\) the bipartition
\(\{a,b,c\}\sqcup\{x,y,z\}\), orient its edges from the first
part to the second, and let \(A=A_{K_{3,3}}\) be the reduced incidence
matrix with row order \(a,b,c,y,z\) and column order
\[
 ax,ay,az,bx,cx,by,cy,bz,cz.
\]
These edges correspond, in order, to the listed elements of \(R\).
The first five edges form a spanning tree \(T\), and
\(\det A_T=1\).  With
\[
 D_{\mathrm{row}}=\diag(1,1,1,1,-1),\qquad
 D_{\mathrm{col}}=\diag(1,1,1,1,-1,1,-1,1,-1),
\]
direct verification gives
\[
 LK_R=D_{\mathrm{row}}A_T^{-1}A_{K_{3,3}}D_{\mathrm{col}}.
\]
Thus \(LK_R\) is projectively equivalent to \(A_{K_{3,3}}\).
Since \(\ker(L)=\col(K_C)\), the linear-dependence argument in
the proof of Lemma~\ref{prop:u24-contraction} gives
\[
 (M[K]/C)|R=M[LK_R]\cong M(K_{3,3}).
\]
The product is not cographic by Fact~\ref{fact:cographic-excluded}.
\end{proof}

\begin{proof}[Proof of Proposition~\ref{prop:specified-minors}\textup{(i)} for \((C_4,K_{2,3})\)]
Put \(K=\Delta_4\otimes A_{K_{2,3}}\), so that
\(\rank(K)=12\) and \(M[K]=M_{\F}(C_4,K_{2,3})\).
If \(\operatorname{char}(\F)\ne2\), choose
\[
\begin{aligned}
 C&=\{(1,12),(3,22),(3,23),(4,23),(4,12),(4,21),(2,13),(2,22),(3,11),(1,11)\},\\
 R&=\{(2,21),(2,11),(1,13),(2,12)\}.
\end{aligned}
\]
Take
\[
 L=\begin{bmatrix}
1&0&1&1&0&0&-1&0&0&0&0&0\\
0&-1&0&0&0&1&-1&1&1&0&1&0
\end{bmatrix}.
\]
Then
\[
 \rank(K_C)=10,\qquad \rank(L)=2,\qquad LK_C=0.
\]
The minor of \(K_C\) on rows \(2,3,4,5,7,8,9,10,11,12\)
is \(1\).  The six column-pair determinants of \(LK_R\) are
\(-1,-2,-1,-1,-1,-1\), all nonzero in \(\F\).  Hence
\((M[K]/C)|R\cong U_{2,4}\) by Lemma~\ref{prop:u24-contraction}.

In characteristic two, choose
\[
\begin{aligned}
 C&=\{(1,11),(2,22),(3,13),(2,21),(3,21),(1,12),(4,22),(4,13),(1,23)\},\\
 R&=\{(3,12),(3,11),(1,13),(2,23),(4,11),(4,12),(2,11)\}.
\end{aligned}
\]
Take
\[
 L=\begin{bmatrix}
0&0&0&0&0&1&1&1&0&0&0&0\\
1&0&1&1&1&0&0&0&0&1&1&0\\
1&0&1&1&1&0&0&0&0&0&0&1
\end{bmatrix}.
\]
Here
\[
 \rank(K_C)=9,\qquad \rank(L)=3,\qquad LK_C=0.
\]
The minor of \(K_C\) on rows \(1,2,3,4,5,6,7,9,10\) is \(1\),
and the columns of \(LK_R\) are precisely the nonzero vectors of
\(\operatorname{GF}(2)^3\).  Thus
\((M[K]/C)|R\cong F_7\) by Lemma~\ref{prop:f7-contraction}.
The two cases and Fact~\ref{fact:regular-excluded} prove nonregularity.
\end{proof}

\subsection{Three circuit factors}\label{subsec:three-circuit-minor}
\begin{proof}[Proof of Proposition~\ref{prop:specified-minors}\textup{(iii)}]
Use the representation \(K=\Delta_3\otimes\Delta_3\otimes\Delta_3\),
which has rank \(8\).
If \(\operatorname{char}(\F)\ne2\), choose
\[
\begin{aligned}
 C&=\{(1,1,2),(1,2,1),(1,1,3),(3,3,3),(2,2,3),(2,1,1)\},\\
 R&=\{(3,2,2),(3,2,1),(2,3,2),(3,2,3)\}.
\end{aligned}
\]
Take
\[
 L=\begin{bmatrix}
0&0&0&-1&0&1&0&0\\
0&0&0&0&0&0&-1&1
\end{bmatrix}.
\]
We have
\[
 \rank(K_C)=6,\qquad \rank(L)=2,\qquad LK_C=0.
\]
The minor of \(K_C\) on rows \(1,2,3,4,5,7\) is \(1\), and the
six column-pair determinants of \(LK_R\) are \(1,-2,-1,1,1,-1\).
Lemma~\ref{prop:u24-contraction} gives
\((M[K]/C)|R\cong U_{2,4}\).

In characteristic two, choose
\[
\begin{aligned}
 C&=\{(2,3,3),(3,2,3),(3,1,3),(2,2,1),(1,2,2)\},\\
 R&=\{(1,1,3),(2,1,2),(2,1,1),(1,1,2),(1,1,1),(3,1,1),(3,3,1)\}.
\end{aligned}
\]
Take
\[
 L=\begin{bmatrix}
1&1&0&0&0&0&0&0\\
0&0&0&0&1&1&0&0\\
1&0&1&0&1&0&0&1
\end{bmatrix}.
\]
Then
\[
 \rank(K_C)=5,\qquad \rank(L)=3,\qquad LK_C=0.
\]
The minor of \(K_C\) on rows \(1,3,4,5,7\) is \(1\), and the
columns of \(LK_R\) are precisely the seven nonzero vectors of
\(\operatorname{GF}(2)^3\).  Therefore
\((M[K]/C)|R\cong F_7\) by Lemma~\ref{prop:f7-contraction}.
Fact~\ref{fact:regular-excluded} proves the assertion in both cases.
\end{proof}

\section*{Acknowledgements}
This work is supported by the National Natural Science Foundation of China (Grant No. 12571350) and the Guangdong Basic and Applied Basic Research Foundation (Grant Nos. 2026A1515012237 and 2025A1515010457). 
\section*{Declaration of AI use}
During the preparation of this manuscript, the authors used OpenAI's ChatGPT and Codex to assist with English-language editing, \LaTeX{} typesetting, structural organization, and consistency checking of the mathematical exposition.  After using these tools, the authors reviewed and edited the content as needed and take full responsibility for the content of the published article.

\bigskip
\noindent
Houshan Fu\\
School of Mathematics and Information Science, Guangzhou University\\
Guangzhou 510006, Guangdong, P. R. China

\medskip\noindent
Yujiao Ma and Suijie Wang\\
School of Mathematics, Hunan University\\
Changsha 410082, Hunan, P. R. China

\medskip\noindent
Suijie Wang\\
Greater Bay Area Institute for Innovation, Hunan University\\
Changsha 410082, Hunan, P. R. China

\medskip\noindent
Corresponding author: Yujiao Ma.\\
Emails: \href{mailto:fuhoushan@gzhu.edu.cn}{fuhoushan@gzhu.edu.cn} (Houshan Fu);\\
\href{mailto:yujiaoma@hnu.edu.cn}{yujiaoma@hnu.edu.cn} (Yujiao Ma);\\
\href{mailto:wangsuijie@hnu.edu.cn}{wangsuijie@hnu.edu.cn} (Suijie Wang).

\end{document}